\documentclass[a4paper,10pt]{article}

\usepackage[utf8]{inputenc}
\usepackage[T1]{fontenc}

\usepackage[
	backend = biber,
	hyperref = auto,
	backref = true,
	doi = false,
	url = false
]{biblatex}	% Disable DOI and URL for arXiv

\usepackage[colorlinks]{hyperref}
\hypersetup{
	linkcolor=blue,
	citecolor=red,
	urlcolor=teal,
	pdftitle = {Towards generalized prehomogeneous zeta integrals},
	pdfauthor = {Wen-Wei Li}
}

\usepackage{amsthm}
\usepackage{amssymb}
\usepackage{mathtools}
\usepackage{stmaryrd}	\SetSymbolFont{stmry}{bold}{U}{stmry}{m}{n}	% To avoid warnings
\usepackage{bm}
\usepackage{mathrsfs}
\usepackage{xspace}
\usepackage{tikz-cd}
\usetikzlibrary{positioning, shapes.geometric, patterns, graphs, decorations.pathmorphing}

\usepackage{paralist} %\setdefaultitem{$\star$}{}{}{}

\newcommand{\Z}{\ensuremath{\mathbb{Z}}}
\newcommand{\Q}{\ensuremath{\mathbb{Q}}}
\newcommand{\R}{\ensuremath{\mathbb{R}}}
\newcommand{\CC}{\ensuremath{\mathbb{C}}}

\newcommand{\gr}{\ensuremath{\mathrm{gr}}}

\newcommand{\Nrd}{\ensuremath{\mathrm{Nrd}\xspace}}	% Reduced norm
\newcommand{\topwedge}{\ensuremath{\bigwedge^{\mathrm{max}}}}

\newcommand{\dd}{\ensuremath{\;\mathrm{d}}}
\newcommand{\Schw}{\ensuremath{\mathcal{S}}}	% Schwartz space

\newcommand{\lrangle}[1]{\ensuremath{\langle #1 \rangle}}

\newcommand{\Stab}{\ensuremath{\mathrm{Stab}}}

\newcommand{\Hom}{\ensuremath{\mathrm{Hom}}}
\newcommand{\End}{\ensuremath{\mathrm{End}}}
\newcommand{\rightiso}{\ensuremath{\stackrel{\sim}{\rightarrow}}}

\newcommand{\Image}{\ensuremath{\mathrm{Im}\xspace}}
\newcommand{\dotimes}[1]{\ensuremath{\underset{#1}{\otimes}}}

\newcommand{\Ad}{\ensuremath{\mathrm{Ad}\xspace}}
\newcommand{\Spec}{\ensuremath{\mathrm{Spec}\xspace}}
\newcommand{\Gm}{\ensuremath{\mathbb{G}_\mathrm{m}}}

\newcommand{\utimes}[1]{\ensuremath{\overset{#1}{\times}}}
\newcommand{\Supp}{\ensuremath{\mathrm{Supp}}}

\newcommand{\GL}{\ensuremath{\mathrm{GL}}}

\newcommand{\Sp}{\ensuremath{\mathrm{Sp}}}

\theoremstyle{plain}
\newtheorem{proposition}{Proposition}
\newtheorem{lemma}[proposition]{Lemma}
\newtheorem{theorem}[proposition]{Theorem}
\newtheorem{corollary}[proposition]{Corollary}
\theoremstyle{definition}
\newtheorem{definition}[proposition]{Definition}
\newtheorem{definition-theorem}[proposition]{Definition-Theorem}
\newtheorem{definition-proposition}[proposition]{Definition-Proposition}
\newtheorem{hypothesis}[proposition]{Hypothesis}

\newtheorem{axiom}[proposition]{Axiom}
\newtheorem{example}[proposition]{Example}
\theoremstyle{remark}
\newtheorem{remark}[proposition]{Remark}
\numberwithin{equation}{section}

\newcommand{\mmap}{\ensuremath{\bm{\mu}}}	% moment map
\newcommand{\quoted}[1]{\ensuremath{\text{\textquotedblleft} #1 \text{\textquotedblright}}}	% Deligne's quotation mark
\newcommand{\dcate}[1]{\ensuremath{\text{-}\mathsf{#1}}}	% Categories with a dash on the left

\renewcommand{\emptyset}{\ensuremath{\varnothing}}	% Symbol for the emptyset

\usepackage{amsmath}
\usepackage[nottoc]{tocbibind}

\usepackage{geometry}
\usepackage{lmodern}

\numberwithin{proposition}{section}

\title{Towards generalized prehomogeneous zeta integrals}
\author{Wen-Wei Li}
\date{}

\DeclareFieldFormat{postnote}{#1}	% Disable pagination in postnote for biblatex
\begin{document}

\maketitle

\begin{abstract}
	Let $X$ be a prehomogeneous vector space under a connected reductive group $G$ over $\R$. Assume that the open $G$-orbit $X^+$ admits a finite covering by a symmetric space. We study certain zeta integrals involving (i) Schwartz functions on $X$, and (ii) generalized matrix coefficients on $X^+(\R)$ of Casselman--Wallach representations of $G(\R)$, upon a twist by complex powers of relative invariants. This merges representation theory with prehomogeneous zeta integrals of Igusa et al. We show their convergence in some shifted cone, and prove their meromorphic continuation via the machinery of $b$-function together with V.\ Ginzburg's results on admissible $D$-modules. This provides some evidence for a broader theory of zeta integrals associated to affine spherical embeddings.
\end{abstract}

{\scriptsize
\begin{tabular}{ll}
	\textbf{MSC (2010)} & Primary 22E50; Secondary 11F70, 11S90 \\
	\textbf{Keywords} & Zeta integrals, Schwartz spaces, prehomogeneous vector spaces
\end{tabular}}

\tableofcontents

\section{Introduction}
It is becoming clear that many of the zeta integrals arising from the theory of automorphic representations, in both its local and global aspects, can be understood in terms of embeddings of homogeneous spaces. The global setting is discussed by Sakellaridis in \cite{Sak12}. Braverman and Kazhdan \cite{BK00} also tried to extend the Godement--Jacquet theory to more general $L$-functions via reductive monoids, which is also pursued by L.\ Lafforgue \cite{Laf14}. Another attempt at a general local framework can be found in \cite{Li15}. In all these approaches, one considers a connected reductive group $G$ over a suitable field $F$, a \emph{spherical homogeneous $G$-space} $X^+$, meaning that it has an open Borel orbit, together with an equivariant embedding $X^+ \hookrightarrow X$ such that $X$ is an \emph{affine normal} $G$-variety in which $X^+$ is open dense. For the Godement--Jacquet theory, the embedding in question is just $D^\times \hookrightarrow D$ which is $D^\times \times D^\times$-equivariant, where $D$ stands for a central simple $F$-algebra.

Assume $F$ is local of characteristic zero. All these zeta integrals involve a space of test functions on $X(F)$, called the \emph{Schwartz space}. For the Godement--Jacquet case $X$ is a vector space, so the evident candidate is the usual Schwartz--Bruhat space. The general case is much more delicate, and this is related to the singularity of $X$; see \cite{BNS16} for the unramified geometric aspect. Thus a natural approach is to look at the case of smooth $X$ first.

Suppose that $(\rho, X)$ is an algebraic, finite-dimensional representation of $G$. If there is an open dense $G$-orbit $X^+$ in $X$, we say the triplet $(G, \rho,X)$ is a \emph{prehomogeneous vector space}. These objects have been studied in depth by M. Sato and his school. This fits into the previous framework when $X^+$ is spherical, and $X$ is surely smooth. It turns out that this is the typical case of smooth affine embeddings: Luna \cite[pp.98-99]{Lu73} proved that smooth affine spherical $G$-varieties are fibered in prehomogeneous vector spaces. We remark that the prehomogeneous vector spaces with spherical open $G$-orbit have been classified in \cite{BR96,Le98} over algebraically closed field of characteristic zero.

Given $(G, \rho, X)$ as above, the Schwartz space $\Schw(X)$ is available. We also assume $X^+$ is affine. Then $\partial X := X \smallsetminus X^+$ is the union of the zero loci of \emph{basic relative invariants} $f_1, \ldots, f_r \in F[X]$: they are $G$-eigenfunctions, with eigencharacters $\omega_1, \ldots, \omega_r \in \Hom_{\text{alg.grp.}}(G, \Gm)$. The characterization of basic relative invariants will be given in Proposition \ref{prop:relative-invariants}. We set
\[ \Lambda_\Z := \bigoplus_{i=1}^r \Z\omega_i, \quad \Lambda_\R := \Lambda_\Z \otimes \R, \quad \Lambda_\CC := \Lambda_\Z \otimes \CC. \]
Also set $|f|^\lambda := \prod_{i=1}^r |f_i|^{\lambda_i}$ if $\lambda = \sum_{i=1}^r \omega_i \otimes \lambda_i \in \Lambda_\CC$: this is a smooth function on $X(F)$. Write $\Re(\lambda) \gg_X 0$ to mean $\Re(\lambda_i) \gg 0$ for $i=1,\ldots,m$.

On the other hand, consider an irreducible smooth complex representation $(\pi, V_\pi)$ of $G(F)$ in a suitable category. When $F=\R$, the natural choice are the \emph{SAF representations} (smooth admissible Fréchet of moderate growth), also known as Casselman--Wallach representations; see \cite{BK14}. Denote the continuous $\Hom$ space $\mathcal{N}_\pi := \Hom_G(\pi, C^\infty(X^+))$. The zeta integral we envisage takes the form
\[ Z_\lambda(\eta,v,\xi) := \int_{X^+(F)} \eta(v)|f|^\lambda \xi, \quad \eta \in \mathcal{N}_\pi, \; v \in V_\pi, \; \xi \in \Schw(X) \]
where $\lambda \in \Lambda_\CC$ and we assume $\Re(\lambda) \gg_X 0$. Furthermore we take $\xi$ and $\eta(v)$ to have values in \emph{half-densities}, so that their product is a density (i.e.\ measure) whose integral makes sense without making any choice. To do this we have to modify $\mathcal{N}_\pi$ and $\Schw(X)$ to be $\mathscr{L}^{1/2}$-valued, where $\mathscr{L}^{1/2}$ is the $G$-equivariant line bundle of half-densities. This presents no difficulty since $\mathscr{L}^{1/2}$ turns out to be equivariantly trivializable in our case (Lemma \ref{prop:density-trivialization}). This choice also leads to natural normalizations, see Example \ref{eg:Godement-Jacquet}.

When $(G, \rho, X)$ is as in Godement--Jacquet theory, then $r=1$ and $f \in F[X]$ can be taken to be the reduced norm. Moreover, $\dim_\CC \mathcal{N}_\pi \leq 1$ with equality only when $\pi \simeq \tau \boxtimes \check{\tau}$, in which case a generator $\eta$ is given by matrix coefficients. Hence we get the familiar Godement--Jacquet integral up to some shift caused by half-densities. When $(G, \rho, X)$ is general but $\pi = \text{triv}$, we obtain the \emph{prehomogeneous zeta integrals} studied by M.\ Sato, T.\ Shintani (see \cite{Sa89,Ki03}), Igusa \cite{Ig00} et al. In both cases, such integrals encode a wealth of information.

Turning back to the general case, several questions arise immediately.
\begin{enumerate}[(a)]
	\item Convergence of $Z_\lambda(\eta,v,\xi)$ for $\Re(\lambda) \gg_X 0$.
	\item Meromorphic continuation in $\lambda$, it should even be rational when $F$ is non-Archimedean.
	\item Continuity in $v$ and $\xi$, when $F$ is Archimedean.
\end{enumerate}
For the Godement--Jacquet integrals, (a) and (b) have been solved in \cite{GJ72}, but the $v$ and $\xi$ are somehow constrained for Archimedean $F$. For the prehomogeneous zeta integrals ($\pi=\text{triv}$), all these properties are established (see \cite{Ki03,Sa89}). When $F$ is non-Archimedean, $G$ is split and $X^+$ is wavefront, (a) and (b) are verified in \cite[\S 6.2]{Li15}: the proof uses the asymptotics for generalized matrix coefficients $\eta(v)$ in \cite{SV12}, the Cartan decomposition for $X^+$ as well as Igusa's theory.

One may even consider the functional equation with respect to the Fourier transform $\mathcal{F}: \Schw(X) \rightiso \Schw(\check{X})$. General speculations can be found in \cite{Li15}, but this is beyond the scope of this article.

Let us move to the main topic of this article --- the case $F=\R$.
\begin{hypothesis}\label{hyp:overall}
	Take $F=\R$ and assume that $(G, \rho,X)$ is a prehomogeneous vector space such that $G$ is a connected reductive group and the open $G$-orbit $X^+$ admits a finite equivariant covering by a symmetric space $X^+_0$ under $G$. In this case we say $X^+$ is an \emph{essentially symmetric space}.
\end{hypothesis}
Tools from harmonic analysis are thus available to us. For example, in the Godement--Jacquet case $X^+$ is just the ``group case'' of symmetric spaces. More instances can be found in \S\ref{sec:PVS}. The main results are summarized below.

\begin{theorem}
	Under the Hypothesis \ref{hyp:overall}, fix an irreducible SAF representation $(\pi, V_\pi)$ and $\eta \in \mathcal{N}_\pi$.
	\begin{enumerate}[(i)]
		\item There exists $\kappa = \sum_{i=1}^r \omega_i \otimes \kappa_i \in \Lambda_\R$, depending solely on $\pi$, such that $Z_\lambda(\eta, v, \xi)$ is convergent whenever $\Re(\lambda_i) \geq \kappa_i$ for $i=1, \ldots, m$. It defines a holomorphic function in $\lambda$ with values in the continuous dual $(V_\pi \hat{\otimes} \Schw(X))^\vee$ in the range of convergence.
		\item One can choose a holomorphic function of the form $L(\eta,\lambda) = \prod_{i=1}^m \Gamma(\alpha_i(\lambda))^{-1}$ where $\alpha_i$ are affine functions on $\Lambda_\CC$ with $\Q$-rational gradients, such that
		\[ LZ_\lambda(\eta,v,\xi) := L(\eta,\lambda) Z_\lambda(\eta,v,\xi) \]
		can be holomorphically extended to all $\lambda \in \Lambda_\CC$, and this gives rise to a holomorphic function valued in $(V_\pi \hat{\otimes} \Schw(X))^\vee$.
		\item For each $\lambda$, the pairing $(v,\xi) \mapsto LZ_\lambda(\eta, v, \xi)$ is a $G(\R)$-invariant bilinear form on $(\pi \otimes |\omega|^\lambda) \times \Schw(X)$. Here $|\omega|^\lambda := \prod_{i=1}^r |\omega_i|^{\lambda_i}$.
	\end{enumerate}
\end{theorem}
For complete statements, please consult the Theorems \ref{prop:zeta-convergence}, \ref{prop:meromorphy}, \ref{prop:meromorphy-1}. Holomorphy here is understood in the weak sense, eg. $\lambda \mapsto LZ_\lambda(\eta,v,\xi)$ is holomorphic for each $(v,\xi)$. Therefore one can say $Z_\lambda$ admits a meromorphic continuation with a ``denominator'' $L(\eta,\lambda)$. Note that in the Godement--Jacquet case, this improves upon the original statements in \cite{GJ72}: we obtain continuity properties, whereas $\xi$ is arbitrary and $v \in V_\pi$ is not required to be $K$-finite. Here $K \subset G(\R)$ is a maximal compact subgroup.

We regard prehomogeneous vector spaces mainly as a testing ground. The real aim of this article is to probe for the general techniques for studying generalized Archimedean zeta integrals and locate the essential difficulties, thereby giving support to the viewpoint of \cite{Li15}. Some similar zeta integrals have been studied in \cite{BR05}, which is based on case-by-case discussion with explicit computations. By the way, the results here also complement for some missing details in \cite{Li15} in the Archimedean setting.

Let us sketch the techniques. The convergence of $Z_\lambda$ for $\Re(\lambda) \gg_X 0$ is based on an estimate on the generalized matrix coefficients $u := \eta(v)$ that is uniform in $v$; such a result is furnished by \cite{KSS14}. The framework of Nash functions turns out to be a flexible vehicle for such considerations.

As for the meromorphy and the description of the denominator $L(\eta,\lambda)$, we follow the standard approach via Bernstein--Sato $b$-functions. Grosso modo, its principal input is the \emph{holonomicity} of the $D_{X^+_\CC}$-module $\mathscr{M}'$ generated by $u$. The notion of \emph{admissible $D$-modules} introduced by V.\ Ginzburg \cite{Gin89} seems well-suited for this purpose; the assumption that $X^+$ is essentially symmetric intervenes here. To this end we assume $v \in V_\pi^{K\text{-fini}}$ and invoke Ginzburg's result. These observations will yield a meromorphic family of tempered distributions on $X(\R)$ that restricts to $u|f|^\lambda$ on $X^+(\R)$; we refer to \cite{GSS16} for an illustration of relevant ideas and applications. Here we request more: it should coincide with the convergent integral $\xi \mapsto Z_\lambda(\eta,v,\xi)$ when $\Re(\lambda) \gg_X 0$. Fortunately, the necessary arguments have been well explained in \cite[Appendice A]{BD92}. This step applies to a single $v \in V_\pi^{K\text{-fini}}$: it says nothing about continuity in $v$.

Note the the holonomicity in the case $\pi = \text{triv}$ considered by Sato, Shintani et al.\ is immediate. Indeed, $u$ is locally constant in this case, hence $\text{Ch}(\mathscr{M}')$ is nothing but the zero section of $T^* X^+$.

The final step is to extend this family to all $v \in V_\pi$ and show that continuity is not lost. To go from $V_\pi^{K\text{-fini}}$ to $V_\pi$ we make use of action of the Schwartz algebra $\Schw(G)$ on $V_\pi$, observing that the action can also be transposed to $\xi \in \Schw(X)$ in the range of convergence. To propagate continuity to all $\lambda \in \Lambda_\CC$, we prove and apply a mild generalization to the \emph{principle of analytic continuation} of Gelfand--Shilov \cite[Chapter I, A.2.3]{GS1}.

To allow more general $X^+$ in this framework, one has to extend Ginzburg's result on the holonomicity beyond essentially symmetric spaces. We hope to address this issue in the future. In contrast, the case of non-prehomogeneous $X$ is more delicate because a general definition of Schwartz spaces is not yet available.

\subsection*{Organization of this article}
In \S\ref{sec:harmonic-analysis} we collect the necessary preliminaries on generalized matrix coefficients of an SAF representation on an essentially symmetric space. In \S\ref{sec:PVS} we make a brief review of prehomogeneous vector spaces and sets up the geometric framework (Axiom \ref{axiom:PVS}). Several non-trivial examples of prehomogeneous vector spaces are also presented; they should be compared with the examples of \cite[\S 5.2]{Sak12} in the global setting. In \S\S\ref{sec:holonomicity}---\ref{sec:admissibility} the results on $b$-functions and admissible $D$-modules are recast to our purposes. The zeta integral is defined in \S\ref{sec:zeta-integral} and its convergence for $\Re(\lambda) \gg_X 0$ is proved. After the intermezzo \S\ref{sec:Schwartz-action} on Schwartz algebra actions, we prove the meromorphic continuation of zeta integrals for $K$-finite vector in \S\ref{sec:meromorphy}, then extend to all $v \in V_\pi$ and establish its continuity. A useful result from Gelfand--Shilov (Theorem \ref{prop:GS-principle}) is generalized and proved in \S\ref{sec:meromorphy}; it has also been used in \cite{Li15}.

\subsection*{Acknowledgement}
The author is grateful to Jeffrey Adams, Wee Teck Gan, Dihua Jiang, Eitan Sayag and Jun Yu for inspiring conversations. The author also appreciates the referee's pertinent comments.

\subsection*{Conventions}
For a scheme $Z$ over a ring $R$ and an $R$-algebra $S$, we write $Z(S)$ for the set of its $S$-points; when $S=\R, \CC$, it comes with the usual topology. In this article, varieties over a field $F$ are always reduced. The tangent (resp. cotangent) bundle of a smooth variety $Z$ is denoted by $TZ$ (resp. $T^* Z$). The coordinate algebra (resp. function field) of an $F$-variety $Z$ is denoted by $F[Z]$ (resp. $F(Z)$ if $Z$ is irreducible).

Let $G$ be an affine algebraic group over some field $F$. We will write $G^\circ$ for its identity connected component, and write $X^*(G) := \Hom_\text{alg.grp.}(G, \Gm)$ for the character lattice. The symbols $\GL(n)$, $\GL(V)$ denote the general linear groups, where $V$ is a finite-dimensional vector space. We also write $\bm{\mu_m}$ for the group scheme of $m$-th roots of unity. Denote by $Z_G$ (resp. $G_\text{der}$) the center (resp. derived subgroup) of $G$. For an $F$-automorphism $\sigma: G \to G$, its fixed subgroup is denoted by $G^\sigma$. The adjoint action is $\Ad(g): x \mapsto gxg^{-1}$. % If $G$ is a torus then $X_*(G) := \Hom_{\text{alg.grp.}}(\Gm, G)$.

Lie algebras are denoted using $\mathfrak{gothic}$ letters. For a Lie algebra $\mathfrak{g}$ over some field, we write $\mathcal{U}(\mathfrak{g})$ for its universal enveloping algebra and $\mathcal{Z}(\mathfrak{g})$ for the center of $\mathcal{U}(\mathfrak{g})$. The fixed subalgebra under an automorphism $\sigma: \mathfrak{g} \to \mathfrak{g}$ is again denoted as $\mathfrak{g}^\sigma$.

Unless otherwise specified, a $G$-variety means a variety $Z$ with right $G$-action, so that $G$ has a left action on $F[Z]$; the stabilizer of $x$ is denoted as $\Stab_G(x)$. Homogeneous $G$-space means a $G$-variety with a single orbit.

The dual of a vector space $V$ is denoted by $V^\vee$, except that for a Lie algebra $\mathfrak{g}$ we write $\mathfrak{g}^*$; the canonical pairing is written as $\lrangle{\cdot, \cdot}$. For $W \subset V^\vee$ we write $W^\perp := \{ T \in V^\vee : T|_W = 0 \}$. Topological vector spaces over $\CC$ are always locally convex and Hausdorff; for $V$ a topological vector space, $V^\vee$ will denote the continuous dual. When talking about continuous representations of a topological group $\Gamma$, the symbol $\Hom_\Gamma(V_1, V_2)$ will mean the space of continuous intertwining operators $V_1 \to V_2$. The trivial representation is denoted as $\text{triv}$.

We will consider maps of the form
\begin{align*}
	T: \Omega & \longrightarrow V^\vee \\
	\lambda & \longmapsto T_\lambda
\end{align*}
where $\Omega$ is a complex manifold and $V$ is a topological vector space. We say $T$ is holomorphic if so is $\lambda \mapsto T_\lambda(v)$ for every $v \in V$. By a meromorphic family $T: \Omega \to V^\vee$, we mean a map $T: \Omega \to V^\vee$ defined almost everywhere such that locally on $\Omega$, there exists a scalar-valued holomorphic function $L(\lambda)$ whose product with $T$ is holomorphic as above.

Let $\Lambda_\R$ be an $\R$-vector space of finite dimension and $\Lambda_\CC := \Lambda_\R \otimes \CC$. Let $U \subset \Lambda_\R$ be open. A holomorphic function $f: \Re^{-1}(U) \to \CC$ is said to be bounded in vertical strips if for each compact subset $C \subset U$, $f$ is bounded over $\Re^{-1}(C)$.

\section{Harmonic analysis on symmetric spaces}\label{sec:harmonic-analysis}
Throughout this section, we fix a connected reductive group $G$ over a field $F$ of characteristic not two. We will assume $F=\R$ later on. By a \emph{finite covering} of a homogeneous $G$-space $Z$, we mean a finite, $G$-equivariant étale covering $\pi: Z_0 \to Z$ by a homogeneous $G$-space $Z_0$.

\begin{definition}
	In this article, by a \emph{spherical variety} over $F$ we mean a normal $G$-variety $Z$ with an open Borel orbit over an algebraic closure of $F$.
\end{definition}

\begin{definition}\label{def:symm-space}
	Let $X^+ \simeq H \backslash G$ be a homogeneous $G$-space with $X^+(F) \neq \emptyset$. By choosing $x_0 \in X^+(F)$ we may write $X^+ \simeq H \backslash G$ with $H := \Stab_G(x_0)$. If there exists an involution $\theta: G \to G$ such that $(G^\theta)^\circ \subset H \subset G^\theta$, we say $X^+$ is a \emph{symmetric space}.
	
	More generally, if $X^+$ admits a finite covering $\pi: X^+_0 \to X^+$ such that $X^+_0$ is a symmetric space, we say $X^+$ is an \emph{essentially symmetric space}. By choosing a base point in $X^+_0(F)$, we deduce a base point in $X^+(F)$ and the covering may be expressed as $H_0 \backslash G \twoheadrightarrow H \backslash G$ with $(G^\theta)^\circ = H^\circ \subset H_0 \subset H$.
\end{definition}
We also say that the stabilizer $H$ is a symmetric (resp. essentially symmetric) subgroup of $G$. This notion is independent of base point by virtue of the following observation.

\begin{lemma}\label{prop:involution-vs-H}
	Suppose $X^+$ is essentially symmetric and $x \in X^+(F)$ has stabilizer $J$. There exists an involution $\tau: G \to G$ defined over $F$ such that $J^\circ = (G^\tau)^\circ$. If $X^+$ is symmetric then $J \subset G^\tau$.
\end{lemma}
\begin{proof}
	Fix a separable closure $\bar{F}$ of $F$. Let $H = \Stab_G(x_0)$ as before. There exists $g \in G(\bar{F})$ such that $x = x_0 g$, therefore $J = g^{-1}Hg$ over $\bar{F}$. Define the involution over $\bar{F}$
	\[ \tau := \Ad(g)^{-1} \circ \theta \circ \Ad(g). \]
	After base change to $\bar{F}$, we have $G^\tau = g^{-1} G^\theta g$ and $J^\circ = (G^\tau)^\circ$. Furthermore, $H \subset G^\theta \implies J \subset G^\tau$ over $\bar{F}$. It suffices to show $\tau$ is defined over $F$. Obviously $\tau|_{Z_G} = \theta|_{Z_G}$ after base change to $\bar{F}$. On the other hand, $(G_\text{der})^\tau$ has identity connected component equal to that of $J^\circ \cap G_\text{der}$, which is defined over $F$. Using \cite[Proposition 1.6]{HW93} we infer that $\tau|_{G_\text{der}}$ is defined over $F$ as well. As $G = Z_G \cdot G_\text{der}$, we see $\tau$ is defined over $F$.
\end{proof}

Symmetric spaces are affine and spherical, and in this case $H^\circ$ is reductive; see \cite[\S 26]{Ti11}. An essentially symmetric space $X^+$ is also spherical. As $\Gamma := H/H_0$ is finite, we see $X^+ \simeq X^+_0/\Gamma$ is also affine. The Luna--Vust theory for the spherical embeddings therefore applies to $X^+$, upon base-change to an algebraic closure of $F$. Such a general framework is not needed in this article, but it is beneficial to notice that symmetric spaces are \emph{wavefront}, that is, its valuation cone defined by Luna--Vust turns out to be the image of the anti-dominant Weyl chamber; see \cite[Theorem 26.25]{Ti11}. Essentially symmetric spaces are also wavefront since the valuation cone of $X^+_0$ surjects onto that of $X^+$, see \cite[\S 15.2]{Ti11}.

Henceforth we take $F=\R$.

Now enters representation theory. We work within the category of \emph{SAF representations} of $G(\R)$ as studied in \cite{BK14}, also known as \emph{Casselman--Wallach representations}. These are certain continuous representations $\pi$ of $G(\R)$ on a nuclear Fréchet space $V_\pi$ over $\CC$ whose topology is defined by a countable family of semi-norms. The morphisms are the continuous intertwining operators. Choose a maximal compact subgroup $K \subset G(\R)$. Then the $K$-finite vectors $V_\pi^{K\text{-fini}}$ form a dense subspace of $V_\pi$, and $V_\pi^{K\text{-fini}}$ is a Harish-Chandra module. The functor $(\pi, V_\pi) \mapsto (\pi, V_\pi^{K\text{-fini}})$ is an equivalence between the categories $\text{SAF}\dcate{Rep}$ and $\text{HC}\dcate{Mod}$. We refer to \cite{BK14} for detailed discussions.

For any smooth $G(\R)$-equivariant line bundle $\mathscr{L}$ over the manifold $X^+(\R)$. Denote by $C^\infty(X^+, \mathscr{L})$ the space of $C^\infty$-sections of $\mathscr{L}$ over $X^+(\R)$. Topologize $C^\infty(X^+, \mathscr{L})$ in the standard way (see \cite[2.2]{Be88}) to make it into a smooth continuous representation of $G(\R)$.

\begin{definition}
	Let $X^+$ be any homogeneous $G$-space and let $\mathscr{L}$ be a smooth $G(\R)$-equivariant line bundle over $X^+(\R)$. For any irreducible SAF representation $\pi$, define the $\CC$-vector space
	\[ \mathcal{N}_\pi(\mathscr{L}) := \Hom_G\left( \pi, C^\infty(X^+, \mathscr{L}) \right). \]
	If $\mathscr{L}$ is trivial, the abbreviation $\mathcal{N}_\pi$ will be used.
\end{definition}

This space is related to the notion of distinguished representations as follows. Decompose $X^+(\R)$ into a finite number of $G(\R)$-orbits $Y_0, \ldots, Y_m$ and choose $x_i \in Y_i$. Each $Y_i$ is isomorphic to some $H_i(\R) \backslash G(\R)$. By Lemma \ref{prop:involution-vs-H}, the subgroup $H_i$ is essentially symmetric (resp. symmetric) if $H$ is. For each $1 \leq i \leq m$, the stabilizer $H_i(\R)$ acts on the fiber $\mathscr{L}_{x_i}$ via a continuous character $\chi_i$. Frobenius reciprocity for continuous representations (see \cite[2.5]{Be88}) then yields
\[ \mathcal{N}_\pi(\mathscr{L}) \simeq \prod_{i=0}^m \Hom_G\left( \pi, C^\infty(Y_i, \mathscr{L}) \right) \simeq \prod_{i=0}^m \Hom_{H_i}( \pi|_{H_i(\R)}, \chi_i ). \]

\begin{theorem}\label{prop:N-fd}
	The space $\mathcal{N}_\pi(\mathscr{L})$ is finite-dimensional.
\end{theorem}
\begin{proof}
	By the discussion above, it suffices to apply the general result \cite[Theorem A]{KO13} to each $Y_i$.
\end{proof}

We call the elements of $\Image\left[ \mathcal{N}_\pi \otimes V_\pi \to C^\infty(X^+, \mathscr{L}) \right]$ the \emph{generalized matrix coefficients} of $\pi$ on $X^+$. A crucial issue in harmonic analysis is to control their asymptotic behavior. To state the next result, we shall assume that $\mathscr{L}$ is trivializable as an equivariant line bundle over $X^+(\R)$, which covers the cases needed in this article.

\begin{definition}
	A \emph{Nash function} on the manifold of $\R$-points of a smooth $\R$-variety, or any connected component thereof, means a $C^\infty$ semi-algebraic function. See \cite{AG08, BCR98}. Notice that Nash functions form an algebra, and if $p$ is a positive Nash function on a Nash manifold, then so is $p^r$ for every $r \in \Q$.
\end{definition}

This notion will be mainly applied to $G(\R)$-orbits in $X^+(\R)$.

\begin{theorem}\label{prop:asymptotic}
	For any irreducible SAF representation $\pi$ and every $\eta \in \mathcal{N}_\pi$, there exist
	\begin{itemize}
		\item a continuous semi-norm $q$ on $V_\pi$, and
		\item a Nash function $p \geq 0$ on $X^+(\R)$
	\end{itemize}
	that satisfy
	\[ |\eta(v)(x)| \leq q(v) p(x), \quad v \in V_\pi, \; x \in X^+(\R). \]
\end{theorem}
\begin{proof}
	We begin with the case of symmetric $X^+$. Fix one $G(\R)$-orbit $Y \simeq H(\R) \backslash G(\R)$ in $X^+(\R)$. In the proof of \cite[Theorem 5.8]{KSS14}, they constructed a ``weight function'' $w: Y \to \R_{>0}$, then defined the Banach space
	\[ E_w := \left\{ c \in C(Y) : \|c\|_w := \sup_{y \in Y} w(y)|c(y)| < +\infty \right\} \]
	endowed with norm $\|\cdot\|_w$, and showed that
	\begin{align*}
		V_\pi & \longrightarrow E_w \\
		v & \longmapsto \eta(v)|_Y
	\end{align*}
	is continuous. By inspecting the argument in \cite[p.246]{KSS14}, we see that $w = w_0 w_1^{-d}$ with
	\begin{itemize}
		\item $w_0 = w_{U_1}^{r_1} \cdots w_{U_k}^{r_k}$ where $U_i$ are finite-dimensional $\R$-algebraic representations of $G$ with $H$-fixed vector $u_i \neq 0$ and Euclidean norm $\|\cdot\|$, so that $w_{U_i}(g) = \|u_i g \|  > 0$ induce Nash functions on $Y$ (see \cite[(5.6)]{KSS14}), and $r_i \in \R$;
		\item $w_1 > 0$ is made up of $\log(c w'_0)$ for suitable $c \in \R_{>1}$, where $w'_0$ is a function constructed in the manner above that verifies $w'_0 \geq 1$;
		\item $d \in \Z_{\geq 0}$ is defined in \cite[p.235]{KSS14}.
	\end{itemize}
	These data depend only on $\pi$. Therefore $w_1^{-d} \gg (w'_0)^{-1}$, and $w = w_0 w_1^{-d} \gg w_{U_1}^{s_1} \cdots w_{U_k}^{s_k}$ for certain $s_i \in \R$. Hereafter write $\omega_i := w_{U_i}$. The continuity of $V_\pi \to E_w$ entails the existence of a continuous semi-norm $q$ on $V_\pi$ satisfying $|\eta(v)| \leq q(v) \prod_{i=1}^k \omega_i^{-s_i}$. For each $1 \leq i \leq k$, we have
	\[ \dfrac{\omega_i}{1 + \omega_i} < \omega_i <  1 + \omega_i, \]
	noting that both sides are still positive Nash functions on $Y$. Take any integer $t_i$.
	\begin{itemize}
		\item If $-s_i > 0$, we have $\omega_i^{-s_i} < (1 + \omega_i)^{-s_i} < (1 + \omega_i)^{t_i}$ whenever $t_i > -s_i$.
		\item If $-s_i < 0$, we have $\omega_i^{-s_i} < \left(\dfrac{\omega_i}{1 + \omega_i}\right)^{-s_i} < \left(\dfrac{\omega_i}{1 + \omega_i}\right)^{t_i}$ whenever $t_i < -s_i$.
	\end{itemize}
	Taking products yields a Nash function $p$ on $Y$ with $\prod_{i=1}^k \omega_i^{-s_i} < p$. Since there are only finitely many $G(\R)$-orbits, the required estimate on $X^+(\R)$ follows.

	Now assume $X^+$ is essentially symmetric. As before, we consider a single $G(\R)$-orbit written as
	\[ H(\R) \backslash G(\R) \simeq Y \subset X^+(\R), \]
	such that $H^\circ$ is a symmetric subgroup. The main ingredients in \textit{loc.\  cit.} are:
	\begin{enumerate}
		\item Cartan decomposition for $Y \simeq H(\R) \backslash G(\R)$, which also holds in the present case since $H \supset H^\circ$;
		\item results of Hoogenboom on the highest weights of $H^\circ$-spherical representations, see the proof of \cite[Lemma 5.7]{KSS14}.
	\end{enumerate}
	These yield an estimate of the form $|\eta(v)(x)| \leq q(v) \tilde{p}(\tilde{x})$, for all $H^\circ(\R) \backslash G(\R) \ni \tilde{x} \mapsto x$, where $\tilde{p} \geq 0$ is a Nash function on $H^\circ(\R) \backslash G(\R)$. The finite group $H^\circ(\R) \backslash H(\R)$ acts on the left of $H^\circ(\R) \backslash G(\R)$ via $H^\circ(\R)g \mapsto H^\circ(\R) hg$. Averaging $\tilde{p}$ over this action yields another Nash function on $H^\circ(\R) \backslash G(\R)$ that descends to a $C^\infty$ function $p \geq 0$ on $H(\R) \backslash G(\R)$. It remains to show that $p$ is semi-algebraic. Indeed, the graph of $p$ is the projection to $(H(\R) \backslash G(\R)) \times \R$ of the semi-algebraic set
	\[ \left\{ (x, \tilde{x}, t) : x \mapsfrom \tilde{x} \stackrel{p}{\longmapsto} t \right\} \subset (H(\R) \backslash G(\R)) \times (H^\circ(\R) \backslash G(\R)) \times \R. \]
	Projection preserves semi-algebraicity by the Tarski--Seidenberg principle (see \cite[Corollary 2.2.8]{AG08} or \cite[\S 2.2]{BCR98}). This completes the proof.
\end{proof}

It seems possible to generalize this to more general real spherical spaces. Cf.\ \cite[\S 6.3]{KKS15}.

\section{Prehomogeneous vector spaces}\label{sec:PVS}
To begin with, let $F$ be any field of characteristic zero. Fix a connected reductive $F$-group $G$. By convention, we let $G$ act from the right on its representations.

\begin{definition}
	A \emph{prehomogeneous vector space} under $G$ is a finite-dimensional $F$-vector space $X$ together with an algebraic representation $\rho: G \to \GL(X)$ such that $X$ contains an open dense orbit, hereafter denoted by $X^+$. We also write $\partial X := X \smallsetminus X^+$.
\end{definition}
For a systematic introduction to prehomogeneous vector spaces, we recommend \cite{Ki03}. We will often write $xg = x\rho(g)$ for the $G$-action on $X$. Then $G$ acts on $F(X)$ and $F[X]$ by $gf(x) = f(xg)$. Taking contragredient yields the dual datum $(\check{\rho}, \check{X})$, which is not prehomogeneous in general.

\begin{definition}
	Let $(G, \rho, X)$ be a prehomogeneous vector space. A \emph{relative invariant} is an eigenfunction $f \in F(X)^\times$ under $G$; the formula $f(xg) = \omega(g)f$ uniquely determines the eigencharacter $\omega \in X^*(G)$ of $f$. We say $f$ is \emph{non-degenerate} if $f^{-1} \dd f$ gives a birational map $X \dashrightarrow \check{X}$. Prehomogeneous spaces admitting a non-degenerate relative invariant are called \emph{regular}.
\end{definition}

The set of relative invariants forms a multiplicative group. Their eigencharacters form a subgroup $X^*_\rho(G) \subset X^*(G)$. Relative invariants with trivial eigencharacter must be constant.

\begin{proposition}[{\cite[Corollary 2.17]{Ki03}}]\label{prop:det-formula}
	For a regular prehomogeneous vector space $(G, \rho, X)$, the character $(\det\rho)^2: G \to \Gm$ belongs to $X^*_\rho(G)$.
\end{proposition}
This is actually the only application of regularity in this article, but we should emphasize that regularity is crucial for the study of functional equations. See \cite{Sa89} or \cite[\S 6]{Li15}.

\begin{proposition}[\cite{Sa89}]\label{prop:relative-invariants}
	There exist relative invariants $f_1, \ldots, f_r \in F[X]$, with eigencharacters $\omega_1, \ldots, \omega_r$, satisfying
	\begin{compactitem}
		\item every relative invariant takes the form $f = c \prod_{i=1}^r f_i^{d_i}$ with $c \in F^\times$ and $d_i \in \Z$;
		\item $\omega_1, \ldots, \omega_r$ form a basis for the $\Z$-module $X^*_\rho(G)$;
		\item the irreducible components of codimension one of $\partial X$ are precisely the divisors $f_i = 0$, $i=1, \ldots, r$.
	\end{compactitem}
	We call $f_1, \ldots, f_r$ the \emph{basic relative invariants} and label them by the codimension-one components of $\partial X$. Then they are unique up to $F^\times$.
\end{proposition}
When we work over $\R$, some boundary divisors over $\CC$ will be paired by complex conjugation. The corresponding basic relative invariants (resp. eigencharacters) over $\CC$ merge into $f_i \overline{f_i}$ (resp. $\omega_i + \overline{\omega_i}$) over $\R$.

%Denote by $X^*_\rho(G)_{\bar{F}}$ the group attached to $X \times_F \bar{F}$ with $G \times_F \bar{F}$-action, where $\bar{F}$ is an algebraic closure of $F$.
\begin{proposition}[{\cite[Theorem 2.28]{Ki03}}]
	Suppose $X$ is prehomogeneous over $\R$ or $\CC$. If $X^+$ is affine, then $\partial X$ is of codimension one. Furthermore, $X$ is a regular prehomogeneous vector space.
%	Furthermore, if $X^*_\rho(G) = X^*_\rho(G)_{\bar{F}}$, then $X$ is a regular prehomogeneous vector space.
\end{proposition}
\begin{proof}
	This is \cite[Theorem 2.28]{Ki03} when $F = \CC$. In fact, it is shown in \cite[Theorem 2.24]{Ki03} that if $\partial X$ is defined by $f=0$, then $f$ is a non-degenerate relative invariant. As we can ensure $f \in \R[X]$, this result also holds over $\R$.
\end{proof}

In what follows, ``vector field'' means an algebraic vector field on a smooth $F$-variety. The Lie algebra $\mathfrak{g}$ also acts on $X$, still denoted as $\rho$. Each $v \in \mathfrak{g}$ induces a vector field $D_v$ on $X$: as a differential operator, it maps any $\varphi \in F[X]$ to the regular function
\[ x \longmapsto \lrangle{ \underbracket{\dd \varphi(x)}_{\in X^\vee}, \;\underbracket{x \cdot \rho(v)}_{\in X} }. \]
Since $(G, \rho, X)$ is prehomogeneous, the vectors $\{D_v : v \in \mathfrak{g} \}$ generate the tangent sheaf $\mathscr{T}_X$ at each $x \in X^+$.

\begin{lemma}\label{prop:diff-op-relation}
	Assume that $\partial X$ is the zero locus of some $f \in F[X]$. Let $v_1, \ldots, v_n \in \mathfrak{g}$ be a basis. Then every vector field $D \in \Gamma(X, \mathscr{T}_X)$ on $X$ satisfies
	\[ f^k D = \sum_{i=1}^n a_i D_{v_i}, \quad a_i \in F[X] \]
	for some $k \geq 0$.
\end{lemma}
\begin{proof}
	Consider the short exact sequence of coherent sheaves on $X^+$.
	\[\begin{tikzcd}[row sep=tiny]
		0 \arrow{r} & \mathscr{I} \arrow{r} & \mathscr{O}_{X^+}^{\oplus n} \arrow{r} & \mathscr{T}_{X^+} \arrow{r} & 0 \\
		& & (b_1, \ldots, b_n) \arrow[mapsto]{r} & \sum_{i=1}^n b_i D_{v_i}. &
	\end{tikzcd}\]
	Since $X^+$ is affine, $H^1(X^+, \mathscr{I})=0$ so we obtain a surjection $\Gamma(X^+, \mathscr{O}_X)^{\oplus n} \twoheadrightarrow \Gamma(X^+, \mathscr{T}_X)$. Thus on $X^+$ we have $D = \sum_{i=1}^n b_i D_{v_i}$ for certain $b_i \in F[X^+]$. As $F[X^+] = F[X][f^{-1}]$, clearing denominators gives $f^k D = \sum_{i=1}^n a_i D_{v_i}$ for some $a_i \in F[X]$ and $k \geq 0$. This equality extends to $X$ by density.
\end{proof}

\begin{axiom}\label{axiom:PVS}
	In this article, we shall assume $F=\R$ and work with a prehomogeneous vector space $(G, \rho, X)$ such that
	\begin{center}
		the open $G$-orbit $X^+$ is an essentially symmetric space, see Definition \ref{def:symm-space}.
	\end{center}
%		\item $X^*_\rho(G) = X^*_\rho(G)_{\CC}$, and the basic relative invariants are all defined over $\R$,
	In this case, $(G, \rho, X)$ is regular and $\partial X$ is the union of the zero loci of basic relative invariants; Lemma \ref{prop:diff-op-relation} is also applicable in this setting.
\end{axiom}

\begin{remark}
	If we assume furthermore that $G$ is split, then the datum $X \supset X^+$ meets the requirements in \cite[Axiom 2.4.3]{Li15} (see also Theorem 6.2.6 in \textit{op.\ cit.}) and it makes sense to define zeta integrals. This will be the topic of \S\ref{sec:zeta-integral}.
\end{remark}

It is the last condition which affords the main geometric input. More generally, we are interested in prehomogeneous vector spaces whose open orbit is spherical, i.e.\ with an open Borel orbit over the algebraic closure. By the general theory of spherical varieties \cite[Theorem 1.2]{Le98}, a linear representation $\rho: G \to \GL(X)$ makes $X$ into a spherical $G$-variety if and only if
\[ \CC[X] \; \text{is multiplicity-free as a $G$-representation}. \]
This problem falls naturally in the realm of invariant theory. A complete classification of such representations $(G, \rho, X)$ over $\CC$ has been obtained independently in \cite{Le98, BR96}, both generalize the earlier work by V. G. Kac \cite{Kac80}. Let us extract some interesting cases from their lists. In the following instances $X^*_\rho(G)$ always has rank one, and the condition $F=\R$ will be immaterial.

\begin{example}\label{eg:Godement-Jacquet}
	Let $F$ be any field. Take $G = \GL(n) \times \GL(n)$, acting on the space of matrices $X := \text{Mat}_{n \times n}$ by
	\[\begin{tikzcd}
		A \arrow[mapsto]{r}[above]{(g_1, g_2)} & g_2^{-1} A g_1.
	\end{tikzcd}\]
	The open orbit is $X^+ = \GL(n)$, which is a standard example of symmetric space (the ``group case''). In this case, $\det$ generates the group of relative invariants; consequently $X^*_\rho(G)$ is generated $(g_1, g_2) \mapsto \det(g_2)^{-1} \deg(g_1)$.
	
	More generally, we may replace $\GL(n) \times \GL(n)$ by $D^\times \times D^\times$ where $D$ is a central simple algebra of dimension $n^2$, and $X = D$. The basic relative invariant can be taken to be $\Nrd$, the reduced norm of $D$. This is the geometric backdrop of Godement--Jacquet theory.
	
	The $X$ in this example carries the structure of a \emph{reductive monoid}, and this is essentially the only example of smooth reductive monoids. See \cite[Theorem 27.25]{Ti11}.
\end{example}

\begin{example}
	Assume $\text{char}(F) \neq 2$. Let $G = \GL(n)$ act on the space $X$ of symmetric bilinear forms on an $n$-dimensional space, namely by
	\[\begin{tikzcd}
		B(x,y) \arrow[mapsto]{r}[above]{g} & \left[ (x,y) \mapsto B(gx, gy) \right].
	\end{tikzcd}\]
	The open orbit $X^+$ consists of non-degenerate $B$; it is a symmetric space whose stabilizer subgroups are orthogonal groups. Furthermore, $\partial X$ is the zero locus of determinant if we express these forms as symmetric $n \times n$ matrices. This prehomogeneous vector space has been studied in depth by Shintani \cite{Shi75}.
\end{example}

\begin{example}
	Similar to the previous example, but consider the $\GL(2n)$-action on alternating forms instead. The stabilizers are then isomorphic to $\Sp(2n)$.
\end{example}

\begin{example}
	Assume $\text{char}(F) \neq 2, 3$. Let $X$ be an \emph{Albert algebra} over $F$. This is a special flavor of Jordan algebras of dimension $27$,  non-associative with identity $1_X$, and we refer to \cite[\S 5]{SV00} for a detailed account. To such an algebra $X$ is attached a cubic form $N: X \to F$, called the norm of $X$ (called the determinant map in \cite[p.120]{SV00}), satisfying $N(1_X)=1$. Let
	\[ G^\flat := \left\{g \in \GL(X) : \forall x, \; N(xg) = N(x) \right\}, \quad G := \Gm \times G^\flat. \]
	Then $G$ acts on $X$ by the representation $\rho$:
	\[\begin{tikzcd}
		x \arrow[mapsto]{r}[above]{{(t,g) \in G}} & tx \cdot g.
	\end{tikzcd}\]
	We claim that
	\begin{enumerate}[(i)]
		\item $G^\flat$ is a simply connected reductive group of type $\mathsf{E_6}$;
		\item $N$ is an irreducible cubic homogeneous polynomial;
		\item as a group scheme, the stabilizer $\Stab_G(1_X)$ equals
			\[ H := \left\{ (t,g) \in \bm{\mu}_3 \times G^\flat: 1_X \cdot g = t^{-1} 1_X \right\}, \]
			and its identity component $H^\circ \subset G^\flat$ equals the automorphism group of the Jordan algebra, $H^\circ$ is a group of type $\mathsf{F_4}$;
		\item $X^+ := \{x \in X: N(x) \neq 0 \}$ is the open $G$-orbit, and it admits a finite covering by the symmetric $G$-space $\Gm \times (H^\circ \backslash G^\flat)$ with kernel $\bm{\mu}_3$;
		\item $N$ is a basic relative invariant with eigencharacter $\omega$ equal to $(t,g) \mapsto t^3$, and $\omega$ generates $X^*_\rho(G)$.
	\end{enumerate}
	Axiom \ref{axiom:PVS} is therefore verified. These assertions are more or less well-known; many of them have actually been recorded in \cite[Example 2.27]{Ki03}. Nonetheless, it seems difficult to find a brief account in sufficient generality, whence the explanations below.

	First, (i) is \cite[Theorem 7.3.2]{SV00}. The irreducibility (ii) of $N$ follows from \cite[Corollary 5.4.6]{SV00}.

	If $t \cdot 1_X g = 1_X$, taking norms yields $t^3 = 1$, and the description of $H$ follows. The identity component $H^\circ$ embeds into $\Stab_{G^\flat}(1_X)$. By \cite[5.9.4]{SV00}, $g \in \GL(X)$ is an automorphism of the Jordan algebra if and only if $g$ preserves both $N$ and $1_X$; it follows that \cite[Theorem 7.2.1]{SV00} that $\Stab_{G^\flat}(1_X)$ is a group of type $\mathsf{F_4}$, in particular it is connected. Therefore $H^\circ = \Stab_{G^\flat}(1_X)$, this completes (iii).

	The proof of \cite[Theorem 7.3.2]{SV00} shows that $G^\flat$ acts transitively on $\{N=1\}$ after base-change to the algebraic closure. Hence $G = \Gm \times G^\flat$ acts transitively on $X^+ := \{N \neq 0\}$, which proves half of (iv); we also infer that $H/H^\circ \simeq \bm{\mu}_3$. In \cite[Proposition 7.3.1]{SV00}, $H^\circ$ is described as the subgroup fixed by an involution $g \mapsto \tilde{g}$ of $G^\flat$. Now we see $H^\circ \backslash G = \Gm \times (H^\circ \backslash G^\flat)$ is symmetric, the covering $H^\circ \backslash G \to H \backslash G$ has kernel $\bm{\mu}_3$. This proves (iv).
	
	Finally, the irreducibility of $N$ forces it to be a basic relative invariant. The corresponding eigencharacter is $(t,g) \mapsto t^3$ since $N$ is cubic.
	
	This prehomogeneous vector space has been studied in \cite{SF84,Mu89}.
\end{example}

% Mention the almost diagonal spaces (informed by J. Adams) [Lu11]

\section{Holonomic \texorpdfstring{$D$}{D}-modules: \texorpdfstring{$b$}{b}-functions}\label{sec:holonomicity}
To begin with, we let $\Bbbk$ be a field of characteristic zero. Let $Z$ be a smooth $\Bbbk$-variety. The theory of algebraic $D$-modules furnishes
\begin{itemize}
	\item a sheaf of $\mathscr{O}_Z$-algebras $\mathscr{D}_Z$, consisting of germs of algebraic differential operators on $Z$;
	\item $D_Z := \Gamma(Z, \mathscr{D}_Z)$, the algebra of algebraic differential operators on $Z$.
\end{itemize}
Hereafter we assume $\Bbbk$ algebraically closed. We say a $\mathscr{D}_Z$-module is \emph{coherent} if it is locally of finite presentation over $\mathscr{D}_Z$; the sheaf $\mathscr{D}_Z$ itself is known to be coherent on both sides. Whenever $Z$ is affine, the pair of functors $\mathscr{F} \mapsto \Gamma(Z, \mathscr{F})$ and $\mathscr{D}_Z \dotimes{D_Z} M \mapsfrom M$ identifies the abelian categories $\mathscr{D}_Z\dcate{Mod}$ and $D_Z\dcate{Mod}$. Coherent $\mathscr{D}_Z$-modules correspond to $D_Z$-modules of finite type using this dictionary. In fact, this equivalence holds for the so-called $\mathscr{D}$-affine varieties, including the flag varieties $P \backslash G$ as well.

Suppose $\mathscr{F}$ is a coherent $\mathscr{D}_Z$-module. Its \emph{characteristic variety} $\text{Ch}(\mathscr{F}) \subset T^* Z$ is defined by first taking a good filtration $F_\bullet \mathscr{F}$, then taking the support of $\gr_F \mathscr{F}$; it is a closed conic subvariety of $T^* Z$. Bernstein's inequality asserts that $\dim \text{Ch}(\mathscr{F}) \geq \dim Z$. If we also have $\dim \text{Ch}(\mathscr{F}) \leq \dim Z$, the coherent $\mathscr{D}_Z$-module $\mathscr{F}$ is said to be \emph{holonomic}. Holonomic modules form a thick subcategory of $\mathscr{D}_Z\dcate{Mod}$.

More generally, we say a $\mathscr{D}_Z$-module $\mathscr{F}$ is \emph{subholonomic} if $\mathscr{F}$ is coherent over $\mathscr{D}_Z$ and $\dim\text{Ch}(\mathscr{F}) \leq \dim Z + 1$.

Let $Z$ be a smooth $\R$-variety, we shall write $Z_\CC := Z \times_\R \CC$. Any generalized function (or distribution if we fix a measure) $u$ on $Z(\R)$ generates a $\mathscr{D}_{Z_\CC}$-module: on each open affine $U$, the differential operators in $\Gamma(U, \mathscr{D}_{Z_\CC})$ act on $u$.

In what follows, we review several standard results on $\mathscr{D}$-modules. The main references are \cite[Appendice A]{BD92} and \cite{D-mod}. Let $X$ be a smooth $\R$-variety and $f \in \CC[X_{\CC}]$ be non-constant. Introduce the symbols $s$, $f^s$. Define the sheaf of algebras $\mathscr{D}_{X_\CC}[s] := \mathscr{D}_{X_\CC} \otimes \CC[s]$ over $X_{\CC}$, which acts on the left module
\[ \mathscr{D}_{X_\CC}[s] f^{s + \Z} = \bigcup_{m \geq 0} \mathscr{O}_{X_\CC}[s] \otimes f^{s-m} \]
in the natural manner, namely: a tangent vector $v$ maps
\[ g s^k \otimes f^{s-m} \longmapsto (vg)s^k \otimes f^{s-m} + g (s-m) s^k (vf) \otimes f^{s-m-1}, \]
where $g$ stands for a local section of $\mathscr{O}_{X_\CC}$. Following Kashiwara, define $\mathscr{N}_f$ as the $\mathscr{D}_{X_\CC}[s]$-submodule generated by $f^s$.

Consider a holonomic $\mathscr{D}_{X_\CC}$-module $\mathscr{M}$. Then $\mathscr{M} \dotimes{\mathscr{O}_{X_\CC}} \mathscr{N}_f$ is a $\mathscr{D}_{X_\CC}[s]$-module under the ``diagonal'' action. For every $u \in \Gamma(X_\CC, \mathscr{M})$, we investigate the $\mathscr{D}_{X_\CC}[s]$-submodule it generates:
\[ \mathscr{A} := \mathscr{D}_{X_\CC}[s] \cdot (u \otimes f^s). \]

As explicated in \textit{loc.\  cit.}, Kashiwara proved the following
\begin{theorem}
	Define $\mathscr{A}$ as above. Then
	\begin{enumerate}
		\item $\mathscr{A}$ is subholonomic as a $\mathscr{D}_{X_\CC}$-module;
		\item for all $\alpha \in \CC$, the specialization $\mathscr{A}_\alpha := \mathscr{A} \dotimes{\CC[s]} \CC[s]/(s-\alpha)$ is a holonomic $\mathscr{D}_{X_\CC}$-module.
	\end{enumerate}
\end{theorem}

In our context, the $\mathscr{D}_{X_\CC}$-module $\mathscr{M}$ arises in the following manner.
\begin{compactitem}
	\item $X$ is smooth affine and $j: X^+ \hookrightarrow X$ is an open immersion;
	\item the manifold $X(\R)$ is equipped with a measure (more precisely, with a smooth density, see \S\ref{sec:zeta-integral});
	\item there is a non-constant $f \in \R[X]$ such that $f > 0$ on $X^+(\R)$ and $\partial X := X \smallsetminus X^+$ is defined by $f = 0$, in particular $X^+$ is affine as well;
	\item $\mathscr{M} := j_* \mathscr{M}'$ where $\mathscr{M}'$ is a holonomic $\mathscr{D}_{X^+_\CC}$-module.
\end{compactitem}
Since $X^+$ is affine, a standard result about $\mathscr{D}$-modules (see \cite[p.244 and p.292]{D-mod}) asserts that $j_* \mathscr{M}'$ is indeed holonomic as a $\mathscr{D}_{X_\CC}$-module. If $u \in \Gamma(X^+_\CC, \mathscr{M}')$ generates a submodule $\simeq D_{X^+_\CC}/I$, then $u$ generates a submodule $\simeq D_{X_\CC}/J$ of $\Gamma(X_\CC, j_* \mathscr{M}')$ where $J := I \cap D_{X_\CC}$; we conclude that the latter submodule is still holonomic. % Originally: "a result of Kashiwara and Kawai..." as in Brylinski--Delorme.

Next, let $X^+ \hookrightarrow X \hookleftarrow \{f=0\}$ be as above. Let $u \in C^\infty(X^+(\R))$. To begin with, we assume that $u$ is supported on a connected component $\Omega$ of $X^+(\R)$ for the usual topology.
\begin{hypothesis}\label{hyp:cond-u}
	Suppose that the tangent sheaf $\mathscr{T}_{X_\CC}$ is generated by global sections $v_1, \ldots, v_l$. For every $(i_1, \ldots, i_l) \in \Z_{\geq 0}^l$, there exists $k \in \Z$ such that every $o \in X(\R)$ has a neighborhood $V$ together with a constant $C_V > 0$, such that
	\[ x \in V \cap X^+(\R) \implies \left| (v_1^{i_1} \cdots v_l^{i_l} \cdot u)(x) \right| \leq C_V f(x)^{-k}. \]
	The point is to verify the inequality in the case $o \in (\partial X)(\R)$.
\end{hypothesis}

For $u$ as above, we infer that for every $N \geq 0$, there exists $k > 0$ such that whenever $\Re(\alpha) \geq k$, the extension by zero of the function $u f^\alpha$ to $X(\R)$ is of class $C^N$. Since the measure is chosen, $uf^\alpha$ can be viewed as a generalized function (or distribution) on $X(\R)$ of order zero. If we strengthen the Hypothesis \ref{hyp:cond-u} to
\begin{equation}\label{eqn:cond-u-strong}
	x \in X^+(\R) \implies \left| (v_1^{i_1} \cdots v_l^{i_l} \cdot u)(x) \right| \leq p(x) f(x)^{-k}
\end{equation}
where $p \geq 0$ is a Nash function on $X(\R)$ (thus locally bounded), then $uf^\alpha$ is even tempered for $\Re(\alpha) \gg 0$, that is, a continuous linear functional on the Schwartz space of $X(\R)$ defined in \cite{AG08}.

We have to impose Hypothesis \ref{hyp:cond-u} on every component $\Omega$. Consider the following procedure.
\[\begin{tikzcd}[every arrow/.style={draw, line width=1.3pt, -Latex}, row sep=small]
	u \in C^\infty(X^+(\R)) \arrow{r}[above, inner sep=0.6em]{\text{generates}} & \mathscr{M}' \arrow{r}[above]{j_*} & \mathscr{M} \ni u \arrow{r} & \mathscr{A} \ni u \otimes f^s \\
	\text{\parbox{70pt}{assume \\ Hypothesis \ref{hyp:cond-u}}} & \text{\parbox{50pt}{assume \\ holonomic}} & \mathscr{D}_{X_\CC}\dcate{Mod} & \mathscr{D}_{X_\CC}[s]\dcate{Mod}
\end{tikzcd}\]

To each $\alpha \in \CC$, specialization yields the section $\quoted{u f^\alpha}$ of $\mathscr{A}_\alpha$. On the other hand, for $\Re(\alpha)$ sufficiently large, $uf^\alpha$ can be seen as a generalized function on $X(\R)$ of order zero.

\begin{proposition}[{\cite[Proposition A.1]{BD92}}]\label{prop:prolongement-ED}
	In the situation above, there exists $k > 0$ such that if $\alpha \in \CC$ satisfies $\Re(\alpha) > k$, then the annihilator of $\quoted{uf^\alpha}$ in $D_{X_\CC}$ also annihilates $u f^\alpha$. Consequently, $u f^\alpha$ generates a holonomic $D_{X_\CC}$-module as well.
\end{proposition}
\begin{proof}
	We repeat the arguments in \textit{loc.\  cit.} Write $\mathscr{A} = D_{X_\CC}[s]/I$ and recall that it is generated by $u \otimes f^s$. Since $\mathscr{A}$ is coherent over $\mathscr{D}_{X_\CC}$, we may find generators $Q_1(s), \ldots, Q_m(s)$ for $I$ over $D_{X_\CC}$ (cf. \cite[p.652]{BD92}). Hence $\mathscr{A}_\alpha = D_{X_\CC}/(Q_1(\alpha), \ldots, Q_m(\alpha))$. It suffices to show that $Q_i(\alpha)$ annihilates $uf^\alpha$ for $i=1, \ldots, m$, if $\Re(\alpha) \gg 0$.
	
	Over $X^+(\R)$ we clearly have $Q_i(\alpha)(uf^\alpha)=0$; in fact it suffices to check this on $\Omega$. Now
	\begin{compactitem}
		\item take $N$ greater than the orders of $Q_1(s), \ldots, Q_m(s) \in D_{X_\CC}[s]$;
		\item take $k$ so large that $\Re(\alpha) > k$ implies $uf^\alpha$ is of class $C^N$ over $X(\R)$.
	\end{compactitem}
	When $\Re(\alpha) > k$, we see $Q_i(\alpha)(uf^\alpha)$ is represented by a continuous function on $X(\R)$, hence it is zero everywhere.
	
	As a consequence, the $D_{X_\CC}$-module generated by $uf^\alpha$ is a quotient of $\mathscr{A}_\alpha$ whenever $\Re(\alpha) > k$, therefore is holonomic.
\end{proof}

In this final part, let us sketch how this entails the existence of the \emph{$b$-functions} of Bernstein--Sato. Keep the assumptions above and consider the $\mathscr{D}_{X_\CC}$-endomorphisms
\begin{align*}
	s: \mathscr{A} & \longrightarrow \mathscr{A}, \quad \text{multiplication by}\; s, \\
	t: \mathscr{A} & \longrightarrow \mathscr{A}, \quad u \otimes a(s) \otimes f^s \mapsto u \otimes a(s+1) \otimes f^{s+1}.
\end{align*}
Note that $t$ is injective and $[t,s] = t$. Now, as explained in \cite[p.109]{Kas03} in the local analytic setting, because $\mathscr{A}$ is subholonomic, $\mathscr{A}/t\mathscr{A}$ is holonomic. This in turn implies that $\End_{D_{X_\CC}}(\mathscr{A}/t\mathscr{A})$ is finite-dimensional over $\CC$. Denote by $b(s) \in \CC[s]$ the minimal polynomial of $s$ as an endomorphism of $\mathscr{A}/t\mathscr{A}$. We conclude that there exists $P(s) \in D_{X_\CC}[s]$ such that
\begin{equation*}
	P(s) (u \otimes f^s) = b(s) u f^{s-1}.
\end{equation*}
The equation holds when $s$ is specialized to $\alpha \in \CC$; by Proposition \ref{prop:prolongement-ED}, it is also satisfied by the generalized function $u f^\alpha$ on $X(\R)$ when $\Re(\alpha) > k$. This differential equation can be used to extend $u f^\alpha$ meromorphically to all $\alpha \in \CC$. One can also deduce information on the poles of $uf^s$ from the zeros of $b$. 

In \S\ref{sec:meromorphy} we will need a multivariate generalization: $\partial X = \bigcup_{i=1}^r \{f_i=0\}$ where $f_i \in \R[X]$ may be assumed non-negative. Put $\Lambda_A := A^{\oplus r}$ for any commutative ring $A$. Now set
\begin{gather*}
	f^s := \prod_{i=1}^r f_i^{s_i}
\end{gather*}
for $s = (s_1, \ldots, s_r) \in \Lambda_\CC$ (formally). The estimate in Hypothesis \ref{hyp:cond-u} should change into
\begin{gather}\label{eqn:cond-u-multivariate}
	\left| (v_1^{i_1} \cdots v_l^{i_l} u)(x) \right| \leq C_V f(x)^{-k}, \quad \text{for some}\; k = (k_1, \ldots, k_r) \in \Lambda_\Z,
\end{gather}
and similarly for the strengthened version \eqref{eqn:cond-u-strong}. The results are similar. Of course, \eqref{eqn:cond-u-multivariate} can be subsumed into the original version by working with $f_\natural := f_1 \cdots f_r$

The key point is to have a $b$-function in the multivariate case. This has been done in \cite{BD92} by reducing to the $r=1$ case. We record their result below.

\begin{theorem}[{\cite[Théorème A.3]{BD92}}]\label{prop:multivariate-b}
	Let $a = (a_1, \ldots, a_r) \in \Lambda_\Z$ with $a_i \geq 0$ for all $i$. There exist
	\begin{compactitem}
		\item an algebraic differential operator with $r$ parameters $P \in D_X[s_1, \ldots, s_r]$, 
		\item a nonzero polynomial $b(s_1, \ldots, s_r) \in \CC[s_1, \ldots, s_r]$,
	\end{compactitem}
	such that
	\[ P (u \otimes f^s) = b(s_1, \ldots, s_r) \cdot u \otimes f^{s-a}, \quad s = (s_1, \ldots, s_r) \in \Lambda_\CC. \]
	The same equation holds if we suppose $\Re(s_i) \gg 0$ for all $i$, and consider $uf^s$ and $uf^{s-a}$ as generalized functions on $X(\R)$ of order zero.
\end{theorem}

For the applications in \S\ref{sec:meromorphy}, $X(\R)$ will always be equipped with a Radon measure, so $uf^s$ can be viewed as a distribution. The strengthened form \eqref{eqn:cond-u-strong} of Hypothesis \ref{hyp:cond-u} will also hold, thus $uf^s$ will actually be tempered in its domain of meromorphy; see Lemma \ref{prop:Nash-bound} and \ref{prop:u-estimate}.

\section{Holonomic \texorpdfstring{$D$}{D}-modules: admissibility}\label{sec:admissibility}
The results below are due to Ginzburg \cite{Gin89}, but we have to cast them into a form suitable for our purposes. As before, let us begin with an algebraically closed field $\Bbbk$ of characteristic zero, and $G$ denotes a connected reductive $\Bbbk$-group. The dual $\mathfrak{g}^*$ of $\mathfrak{g}$ carries the co-adjoint action of $G$.

Let $Z$ be an affine $\Bbbk$-variety with right $G$-action. Each $v \in \mathfrak{g}$ induces a vector field on $Z$; pairing with differential forms gives rise to the \emph{moment map}
\[ \mmap: T^* Z \longrightarrow \mathfrak{g}^*. \]
On the other hand, every $D_Z$-module is endowed with a $\mathcal{U}(\mathfrak{g})$-module structure in this manner. Hereafter, suppose $Z$ is a homogeneous $G$-space with base point $x_0$; set $H = \Stab_G(x_0)$. The cotangent space at $x_0$ is thus identified with $\mathfrak{h}^\perp \subset \mathfrak{g}^*$, and the moment map becomes
\begin{equation}\label{eqn:mmap}\begin{aligned}
	\mmap: T^* Z \rightiso \mathfrak{h}^\perp \utimes{H} G & \longrightarrow \mathfrak{g}^* \\
	[\omega, g] & \longmapsto g^{-1} \omega g.
\end{aligned}\end{equation}
Here $\mathfrak{h}^\perp \utimes{H} G$ means the \emph{contracted product}: we impose the relation $[\omega, hg] = [h^{-1} \omega h, g]$ for all $h \in H$. These arrows are all $G$-equivariant.

In what follows, we use the Killing form on $\mathfrak{g}_\text{der}$ to embed the nilpotent cone $\mathfrak{g}_\mathrm{nil}$ into $\mathfrak{g}^*$. Its image equals $\{y \in \mathfrak{g}^*: \overline{\Ad(G) y} \ni 0 \}$.

We say an action of an algebra or a group on a $\Bbbk$-vector space $V$ is locally finite, if every $v \in V$ is contained in a finite-dimensional invariant $\Bbbk$-subspace.

Let $\mathfrak{k}$ be a Lie subalgebra of $\mathfrak{g}$ of the form $\mathfrak{k} = \mathfrak{g}^\tau$, where $\tau$ is an involution of $G$.
%Following \cite{Gin89}, by a $(\mathfrak{g}, \mathfrak{k})$-module we mean a finitely generated $\mathcal{U}(\mathfrak{g})$-module $V$ such that the $\mathcal{U}(\mathfrak{k})$-action is locally finite.
\begin{definition}[Cf.\ {\cite[1.2]{Gin89}}]
	Given $\mathfrak{k}$ as above, we say a finitely generated $D_Z$-module $M$ is \emph{admissible} if $M$ is locally finite under both  $\mathcal{U}(\mathfrak{k})$ and $\mathcal{Z}(\mathfrak{g})$.
\end{definition}

Thus an admissible $D_Z$-module $M$ is generated by a finite-dimensional vector subspace $M_0$ that is stable under $\mathcal{U}(\mathfrak{k})$ and $\mathcal{Z}(\mathfrak{g})$.

\begin{proposition}
	Let $M$ be a $D_Z$-module generated by a finite-dimensional vector subspace $M_0$ that is stable under $\mathcal{U}(\mathfrak{k})$ and $\mathcal{Z}(\mathfrak{g})$. Then
	\[ \mathrm{Ch}(M) \subset \mmap^{-1}\left( \mathfrak{g}_\mathrm{nil} \cap \mathfrak{k}^\perp \right). \]
	Furthermore, $\mmap^{-1}\left( \mathfrak{g}_\mathrm{nil} \cap \mathfrak{k}^\perp \right)$ is a Lagrangian subvariety of $T^* Z$ provided that for any nilpotent orbit $\mathcal{O} \subset \mathfrak{g}_\mathrm{nil}$, the subvariety $\mmap^{-1}(\mathcal{O} \cap \mathfrak{k}^\perp)$ is Lagrangian.
\end{proposition}
\begin{proof}
	For the first assertion, the arguments in \cite[Lemma 2.1.2]{Gin89} are reproduced as follows. Take the order filtrations $\mathcal{U}_0(\mathfrak{g}) \subset \mathcal{U}_1(\mathfrak{g}) \subset \cdots$ and $D_{Z,0} \subset D_{Z,1} \subset \cdots$. We use the good filtration $F_i M := D_{Z,i} M_0$ for $M$ to get $\text{Ch}(M)$. Denote by $\mathcal{Z}_+(\mathfrak{g}) \subset \mathcal{Z}(\mathfrak{g})$ the augmentation ideal. Assumptions on the stability of $M_0$ imply that $\mathfrak{k}$ and $\mathcal{Z}_+(\mathfrak{g})$ leave each $F_i M$ stable, cf.\ the proof of Theorem \ref{prop:Lagrangian-criterion-0} below, whilst $\mathfrak{k}, \mathcal{Z}_+(\mathfrak{g})$ map to $\gr_{\geq 1} \mathcal{U}(\mathfrak{g})$. Recall that the zero locus of $\gr(\mathcal{Z}_+(\mathfrak{g}))$ (resp. of $\mathfrak{k}$) in $\mathfrak{g}^* = \Spec(\gr(\mathcal{U}(\mathfrak{g})))$ is $\mathfrak{g}_\text{nil} \hookrightarrow \mathfrak{g}^*$ (resp. $\mathfrak{k}^\perp$). From the description \eqref{eqn:mmap}, we infer that $\Supp(\gr_F(M)) \subset \mmap^{-1}\left( \mathfrak{g}_\mathrm{nil} \cap \mathfrak{k}^\perp \right).$

	The second assertion stems from the finiteness of nilpotent orbits.
\end{proof}

In the upcoming applications, the holonomicity of $M$ will be shown by the criteria below.

\begin{proposition}[{\cite[Proposition 1.5.1]{Gin89}}]\label{prop:mmap-nilorbit}
	Let $H, K \subset G$ be closed subgroups and set $Z := H \backslash G$. For every nilpotent orbit $\mathcal{O}$ in $\mathfrak{g}$, the following are equivalent.
	\begin{enumerate}[(i)]
		\item Inside $T^* Z$, the subvariety $\mmap{^{-1}} (\mathcal{O} \cap \mathfrak{k}^\perp)$ is isotropic (resp. co-isotropic, Lagrangian).
		\item Inside $\mathcal{O}$, the subvarieties $\mathcal{O} \cap \mathfrak{h}^\perp$ and $\mathcal{O} \cap \mathfrak{k}^\perp$ are both isotropic (resp. co-isotropic, Lagrangian).
	\end{enumerate}
	Here we refer to the usual symplectic structures on $T^* Z$ and on $\mathcal{O}$ (Kirillov--Kostant--Souriau).
\end{proposition}

\begin{proposition}\label{prop:Lagrangian-criterion}
	Let $\mathfrak{h} := \mathfrak{g}^\theta$ for some involution $\theta: G \to G$. Then $\mathcal{O} \cap \mathfrak{h}^\perp$ is a Lagrangian subvariety of $\mathcal{O}$ for every nilpotent orbit $\mathcal{O} \subset \mathfrak{g}_\mathrm{nil}$.
\end{proposition}
\begin{proof}
	Done in the first part of the proof of \cite[Proposition 3.1.1]{Gin89}.
\end{proof}

Now enters harmonic analysis. Take $G$ to be a connected reductive $\R$-group, $X^+ \simeq H \backslash G$ an affine homogeneous $G$-space. Pick a Cartan involution $\tau$ of $G$ that gives rise to a maximal compact subgroup $K \subset G(\R)$. These objects can be complexified. We feed $\Bbbk = \CC$, $Z = X^+_\CC$ and $H_\CC, K_\CC \subset G_\CC$ into the preceding formalism.

Next, let $\pi$ be an irreducible SAF representation of $G(\R)$. Form the space $\mathcal{N}_\pi := \Hom_G(\pi, C^\infty(X^+))$ as in \S\ref{sec:harmonic-analysis}, as well as the $\mathcal{U}(\mathfrak{g}_\CC)$-stable subspace $V_\pi^{K\text{-fini}}$ that is locally finite under $\mathcal{U}(\mathfrak{k}_\CC)$. Given $\eta \in \mathcal{N}_\pi$ and $v \in V_\pi^{K\text{-fini}}$, the function $u := \eta(v) \in C^\infty(X^+(\R))$ generates a $D_{X^+_\CC}$-module $M$.

\begin{theorem}\label{prop:Lagrangian-criterion-0}
	In the circumstance above, $M := D_{X^+_\CC} u$ is an admissible $D_{X^+_\CC}$-module. It is holonomic if $\mathcal{O} \cap \mathfrak{h}^\perp_\CC$ is a Lagrangian subvariety of $\mathcal{O}$ for any nilpotent orbit $\mathcal{O} \subset \mathfrak{g}_{\mathrm{nil}, \CC}$.
\end{theorem}
\begin{proof}
	For every $\xi \in \mathfrak{g}$, regarded as a vector field on $X^+$, we have
	\[ (\xi P)u = [\xi ,P] u + P(\xi u), \quad P \in D_{X^+_\CC}. \]
	Note that the commutator $[\xi, P]$ is the differential of the $G_\CC$-action on $D_{X^+_\CC}$ (acting by transport of structure) in the $\xi$-direction; this $G_\CC$-action is locally finite since it is algebraically defined. Therefore the action $P \mapsto [\xi,P]$ of $\mathfrak{g}_\CC$ on $D_{X^+_\CC}$ is also locally finite. On the other hand,
	\begin{compactitem}
		\item $\mathcal{Z}(\mathfrak{g}_\CC)$ acts on $u = \eta(v)$ via the infinitesimal character of $\pi$,
		\item $\mathcal{U}(\mathfrak{k}_\CC) \cdot u$ is finite-dimensional since $v \in V_\pi^{K\text{-fini}}$.
	\end{compactitem}
	Hence $M$ is admissible. For any nilpotent orbit $\mathcal{O}$, Proposition \ref{prop:Lagrangian-criterion} entails that $\mathcal{O} \cap \mathfrak{k}^\perp_\CC$ is Lagrangian since $\mathfrak{k}_\CC = \mathfrak{g}_\CC^\tau$. If $\mathcal{O} \cap \mathfrak{h}^\perp_\CC$ is Lagrangian as well, Proposition \ref{prop:mmap-nilorbit} will lead to
	\[ \dim \text{Ch}(M) \leq \dim \mmap^{-1}\left( \mathfrak{g}_\mathrm{nil} \cap \mathfrak{k}^\perp \right) = \dim X^+, \]
	which in turn implies $M$ is holonomic.
\end{proof}

\begin{corollary}\label{prop:u-holonomicity-symm}
	Suppose $X^+$ is essentially symmetric (Definition \ref{def:symm-space}), then the $D_{X^+_\CC}$-module $M$ above is holonomic.
\end{corollary}
\begin{proof}
	In this case $\mathfrak{h}_\CC = \mathfrak{g}^\theta_\CC$ for some involution $\theta: G \to G$. Thus the Proposition \ref{prop:Lagrangian-criterion} can be applied.
\end{proof}

\begin{remark}
	It would be useful to extend the result above to more general $X^+ = H \backslash G$. For example, suppose that $G$ is quasisplit, $H := U$ is a maximal unipotent subgroup. Take a Borel subgroup $B \supset U$ and an opposite $B^- \supset U^-$ of $B$. Using an invariant bilinear form to identify $\mathfrak{g}$ and $\mathfrak{g}^*$, it turns out that $\mathcal{O} \cap \mathfrak{u}^\perp_\CC = \mathcal{O} \cap \mathfrak{b}^-_\CC = \mathcal{O} \cap \mathfrak{u}^-_\CC$ is Lagrangian in $\mathcal{O}$ for any nilpotent orbit $\mathcal{O}$. This is proven in \cite[Theorem 3.3.7]{CG10}. Thus the Proposition \ref{prop:Lagrangian-criterion} is still applicable here. Note that $U \backslash G$ is only quasi-affine.
\end{remark}

\section{Zeta integral: convergence and continuity}\label{sec:zeta-integral}
On any $C^\infty$ manifold $M$ there is a \emph{canonically defined} line bundle $\mathscr{L}_M$, called the \emph{density bundle}. The $C_c$-sections of $\mathscr{L}_M$ can be integrated without reference to any measure. It is deduced from an $\R_{>0}$-torsor on $M$. More generally, one has the bundle of \emph{half-densities} $\mathscr{L}^{1/2}_M$, which also comes from an $\R_{>0}$-torsor and is equipped with an isomorphism $\mathscr{L}^{1/2}_M \otimes \mathscr{L}^{1/2}_M \rightiso \mathscr{L}_M$. As is customary, we write the $C^\infty$ sections of $\mathscr{L}_M$ as $|\Omega|$ where $\Omega$ is a $C^\infty$-differential form of top degree on $M$; similarly, we write $|\Omega|^{1/2}$ for the $C^\infty$ sections of $\mathscr{L}^{1/2}_M$. If a Lie group $\Gamma$ acts on $M$, then $\mathscr{L}_M$ and $\mathscr{L}^{1/2}_M$ have natural $\Gamma$-equivariant structures. We refer to \cite[\S 3.1]{Li15} for details.

Throughout this section, fix a prehomogeneous vector space $(G, \rho, X)$ satisfying the Axiom \ref{axiom:PVS}. We have
\[ X^+ \stackrel{j}{\hookrightarrow} X \hookleftarrow \partial X \]
where $\partial X$ is a hypersurface defined by basic relative invariants $f_1, \ldots, f_r$. Write $\mathscr{L}_X = \mathscr{L}_{X(\R)}$, etc. Clearly $\mathscr{L}_X\big|_{X^+(\R)} = \mathscr{L}_{X^+}$ and $\mathscr{L}_X^{1/2}\big|_{X^+(\R)} = \mathscr{L}^{1/2}_{X^+}$. Remember that $G$ acts on the right of $X$ via $\rho$. By taking transpose, $G$ acts on the left of the dual $\check{X}$, as well as on its exterior powers.

\begin{lemma}\label{prop:density-trivialization}
	The $G(\R)$-equivariant line bundle $\mathscr{L}_{X^+}^{1/2}$ is trivializable. More precisely, choose any $\Omega \in (\topwedge \check{X}) \smallsetminus \{0\}$, there exists a relative invariant $\phi \in \R[X]$ such that $|\phi|^{-1/4} |\Omega|^{1/2}$ is a $G(\R)$-invariant, nowhere-vanishing section of $\mathscr{L}_{X^+}^{1/2}$. Consequently, $\mathscr{L}_{X^+}$ is equivariantly trivializable.
\end{lemma}
\begin{proof}
	Every $g \in G(\R)$ transforms $|\Omega|$ to $|\det\rho(g)| \cdot |\Omega|$. On the other hand, Proposition \ref{prop:det-formula} gives a relative invariant $\phi \in \R[X]$ whose eigencharacter equals $(\det\rho)^2$. Hence $|\phi|^{-\frac{1}{2}} |\Omega|$ trivializes $\mathscr{L}_{X^+}$ equivariantly since $\phi$ is nowhere vanishing on $X^+$. Now take square roots.
\end{proof}
As a byproduct, we obtain a $G(\R)$-invariant measure on $X^+(\R)$ defined by the density $|\phi|^{-\frac{1}{2}} |\Omega|$. On the other hand $|\Omega|$ defines a translation-invariant measure on $X(\R)$.

To the vector space $X$ we have the \emph{$\mathscr{L}_X^{1/2}$-valued Schwartz space} $\Schw(X)$, which is a nuclear Fréchet space with a continuous $G(\R)$-action. Every $\xi \in \Schw(X)$ can be written as a Schwartz section of $\mathscr{L}_X^{1/2}$, namely $\xi = \xi_0 |\Omega|^{\frac{1}{2}}$ for some ordinary Schwartz function $\xi_0$. Half-densities are especially useful when discussing Fourier transforms; see \cite[\S 6.1]{Li15}.

For any irreducible SAF representation $\pi$ of $G(\R)$, the formalism in \S\ref{sec:harmonic-analysis} gives rise to the space $\mathcal{N}_\pi(\mathscr{L}_{X^+}^{1/2})$. In view of Lemma \ref{prop:density-trivialization}, we may safely identify $\mathcal{N}_\pi(\mathscr{L}_{X^+}^{1/2})$ with $\mathcal{N}_\pi$.

Introduce the notation
\[ \Lambda_A :=  X^*_\rho(G) \dotimes{\Z} A, \quad A:\; \text{any commutative ring}. \]
Fix basic relative invariants $f_1, \ldots, f_r \in \R[X]$ with corresponding eigencharacters $\omega_1, \ldots, \omega_r \in X^*_\rho(G)$. For every $\lambda = \sum_{i=1}^r \omega_i \otimes \lambda_i \in \Lambda_\CC$, write
\[ |f|^\lambda := \prod_{i=1}^r |f_i|^{\lambda_i}: \; X^+(\R) \to \CC. \]
For $\kappa = \sum_{i=1}^r \omega_i \otimes \kappa_i \in \Lambda_\R$, we write $\Re(\lambda) \geq_X \kappa$ to indicate that $\Re(\lambda_i) \geq \kappa_i$ for each $i$; the phrase $\Re(\lambda) \gg_X 0$ is similarly interpreted.

\begin{definition}\label{def:zeta-integral}
	Consider the data
	\begin{itemize}
		\item $\pi$: an irreducible SAF representation with underlying space $V_\pi$,
		\item $\eta \in \mathcal{N}_\pi(\mathscr{L}_{X^+}^{1/2}) \simeq \mathcal{N}_\pi$,
		\item $v \in V_\pi$ and $\xi \in \Schw(X)$.
	\end{itemize}
	For $\lambda \in \Lambda_{\CC}$ with $\Re(\lambda) \gg_X 0$, define the corresponding \emph{zeta integral} as
	\[ Z_\lambda(\eta, v, \xi) := \int_{X^+(\R)} \eta(v) |f|^\lambda \xi. \]
	The integral makes no reference to measures since the integrand is a section of $\left( \mathscr{L}_{X^+}^{\frac{1}{2}} \right)^{\otimes 2} \rightiso \mathscr{L}_{X^+}$.
\end{definition}

\begin{remark}\label{rem:zeta-properties}
	Before proving the absolute convergence of the integral, several observations are in order.
	\begin{enumerate}
		\item Granting the absolute convergence for $\Re(\lambda) \gg_X 0$, we see that $Z_\lambda(\eta, v, \xi)$ is multi-linear in $\eta, v, \xi$. It is nontrivial only when $\mathcal{N}_\pi \neq 0$, that is, when $\pi$ is distinguished by some $\Stab_{G(\R)}(x_0)$, $x_0 \in X^+(\R)$. For $\pi=\text{triv}$ we recover the local version of the well-known prehomogeneous zeta integral, whose convergence and meromorphic continuation has been established; see \cite{Sa89,Ki03}.
		\item The basic relative invariants are unique up to scalar. This choice has no effect on the analytic behavior of $Z_\lambda(\eta, v, \xi)$.
		\item Write $|\omega|^\lambda := \prod_{i=1}^r |\omega_i|^{\lambda_i}$, then $v \mapsto \eta(v) |f|^\lambda$ is an element of $\mathcal{N}_{\pi \otimes |\omega|^\lambda}$. Assuming convergence, $(v,\xi) \mapsto Z_\lambda(\eta, v, \xi)$ is then a $G(\R)$-invariant bilinear form $(\pi \otimes |\omega|^\lambda) \times \Schw(X) \to \CC$. Remember that $\Schw(X)$ consists of $\mathscr{L}^{1/2}_X$-valued Schwartz functions, and the $G(\R)$-action must be interpreted accordingly. 
		\item To get a familiar integral of $\CC$-valued functions, one may choose $\Omega \in (\topwedge \check{X}) \smallsetminus \{0\}$ and invoke Lemma \ref{prop:density-trivialization} to write
			\begin{align*}
				\eta(v) & = \eta(v)_0 \cdot |\phi|^{-1/4} |\Omega|^{1/2}, \quad \eta(v)_0 \in C^\infty(X^+(\R)), \\
				\xi & = \xi_0 \cdot |\Omega|^{1/2}, \quad \xi_0: \text{ordinary Schwartz function}.
			\end{align*}
			It has been observed that $\dd\mu := |\phi|^{-1/2} |\Omega|$ is an invariant Haar measure on $X^+(\R)$. As $\phi$ is a relative invariant, $|\phi|^{1/4} = |f|^{\lambda_0}$ for some $\lambda_0 \in \frac{1}{4}\Lambda_\Z$. All in all,
			\begin{equation}\label{eqn:ordinary-zeta} \begin{aligned}
				Z_\lambda(\eta,v,\xi) & = \int_{X^+(\R)} \eta(v)_0 |f|^\lambda |\phi|^{1/4} \xi_0 \cdot |\phi|^{-1/2} |\Omega| \\
				& = \int_{X^+(\R)} \eta(v)_0 |f|^{\lambda + \lambda_0} \xi_0\; \dd\mu.
			\end{aligned}\end{equation}
			This yields a more familiar zeta integral, albeit with a shift $\lambda_0 \in \Lambda_\Q$.
		\item Alternatively, there is a ``flat'' version
			\begin{equation}\label{eqn:flat-zeta}
				Z_\lambda(\eta,v,\xi) = \int_{X(\R)} \eta(v)_0 |f|^{\lambda - \lambda_0} \xi_0 \; \dd\nu
			\end{equation}
			where $\dd\nu$ is the translation-invariant measure on $X(\R)$ determined by $|\Omega|$.
	\end{enumerate}
\end{remark}

For topological vector spaces $V$, $W$, denote by $\mathrm{Bil}(V, W)$ the vector space of jointly continuous bilinear forms $V \times W \to \CC$. It is canonically isomorphic to $(V \otimes W)^\vee$ where $V \otimes W$ comes with the $\pi$-topology \cite[Proposition 43.4]{Tr67}. In the circumstances below $V$ and $W$ will both be nuclear spaces, so one may write $(V \hat{\otimes} W)^\vee$ instead.

\begin{theorem}\label{prop:zeta-convergence}
	There is a $\kappa \in \Lambda_\R$, depending solely on $\pi$, such that the integral in Definition \ref{def:zeta-integral} converges when $\Re(\lambda) \geq_X \kappa$. Furthermore, inside the region of convergence
	\begin{itemize}
		\item $Z_\lambda(\eta, v, \xi)$ is jointly continuous in $(v, \xi)$.
		\item it defines a holomorphic function in $\lambda$ with values in $\mathrm{Bil}(V_\pi, \Schw(X)) \simeq (V_\pi \hat{\otimes} \Schw(X))^\vee$.
		\item For given $(v,\xi)$, it is bounded in vertical strips as a function in $\lambda$.
	\end{itemize}
\end{theorem}
There is no need to discuss the continuity in $\eta$ since $\dim\mathcal{N}_\pi < \infty$. Also recall \cite[\S 41]{Tr67} for various notions of continuity for bilinear maps.
\begin{proof}
	Use the expression \eqref{eqn:flat-zeta}. Theorem \ref{prop:asymptotic} furnishes a continuous semi-norm $q: V_\pi \to \R_{\geq 0}$ together with a Nash function $p$ on $X^+(\R)$, both independent of $v$, such that $\left| \eta(v)_0 \right| \leq q(v) p$. We claim that
	\begin{gather}\label{eqn:ext-Nash-estimate}
		\exists \mu \in \Lambda_\R, \; \exists p_1: \text{Nash function on } X(\R),  \quad |f|^\mu p \leq p_1, \quad .
	\end{gather}
	Granting this property, we have
	\[ |f|^{\Re(\lambda) - \lambda_0} p = |f|^{\Re(\lambda) - \lambda_0 - \mu} |f|^\mu p \leq |f|^{\Re(\lambda) - \lambda_0 - \mu} p_1. \]
	Also observe that $\theta \geq_X 0$ implies that $|f|^\theta$ is bounded by a Nash function on $X(\R)$: indeed, $|f|^\theta \leq \prod_{i=1}^r (1 + |f_i|)^{\theta_i} \leq \prod_{i=1}^r (1 + |f_i|)^{\lceil \theta_i \rceil}$ in this case.
	
	Since $\xi_0$ is a Schwartz function, its product with any Nash function on $X(\R)$ remains bounded \cite[\S 4.1]{AG08}. We deduce that
	\begin{equation}\label{eqn:zeta-convergence-aux}
		\bigg| \eta(v)_0 |f|^{\lambda - \lambda_0} \xi_0 \bigg| \leq q(v) |f|^{\Re(\lambda) - \lambda_0-\mu} p_1 |\xi_0|
	\end{equation}
	is integrable over $X(\R)$ relative to $\dd\nu$ when $\Re(\lambda) \geq_X \lambda_0 + \mu$. The continuity in $\xi$ or $\xi_0$ is easy, whilst the continuity in $v \in V_\pi$ stems from the presence of $q(v)$. All in all, $Z_\lambda(\eta, v, \xi)$ is separately continuous in $v$ and $\xi$. Since $V_\pi$ and $\Schw(X)$ are both Fréchet, joint continuity follows (see \cite[Corollary to Theorem 34.1]{Tr67}). Also, for fixed $(v,\xi)$ it is routine to see the holomorphy of $\lambda \mapsto Z_\lambda(\eta,v,\xi)$ in the range of converge. This amounts to the required holomorphy.
	
	The boundedness on vertical strips for every $(v,\xi)$ is a consequence of \eqref{eqn:zeta-convergence-aux}.

	Finally, to prove \eqref{eqn:ext-Nash-estimate} we appeal to the following Lemma.
\end{proof}

\begin{lemma}\label{prop:Nash-bound}
	Let $p$ be a Nash function on $X^+(\R)$. Then there exist $\mu \in \Lambda_\Z$ with $\mu \geq_X 0$ and a polynomial function $p_1 \geq  0$ on $X(\R)$ satisfying
	\[ |f|^\mu |p| \leq p_1. \]
\end{lemma}
\begin{proof}
	Recall that $X^+$ is the complement of the hypersurface $f_1 \cdots f_r = 0$. Now use the following facts from real algebraic geometry.
	\begin{compactitem}
		\item By \cite[Proposition 2.6.8]{BCR98}, there exists $\mu \in \Lambda_\Z$ such that $\mu \geq_X 0$ and $f^\mu p$ extends to a continuous semi-algebraic function on $X(\R)$.
		\item By \cite[Proposition 2.6.2]{BCR98}, every continuous semi-algebraic function on $X(\R)$ is bounded by a positive polynomial function.
	\end{compactitem}
\end{proof}

\begin{example}
	Consider the Godement--Jacquet case of Example \ref{eg:Godement-Jacquet}, $G = D^\times \times D^\times$, where $D$ is a central simple $\R$-algebra with dimension $n^2$ and reduced norm $\Nrd$. Identify $X^*_\rho(G)$ with $\Z$ by mapping $1$ to $(g_1, g_2) \mapsto \Nrd(g_2)^{-1} \Nrd(g_1)$, which is the eigencharacter of the basic relative invariant $f = \Nrd$. In the notation of Lemma \ref{prop:density-trivialization},
	\[ \det\rho(g_1, g_2) = \Nrd(g_2)^{-n} \Nrd(g_1)^n , \quad \phi = \Nrd^{2n}, \quad |\phi|^{\frac{1}{4}} = |\Nrd|^{\frac{n}{2}}. \]
	Hence $|\phi|^{1/4} = |f|^{n/2}$, i.e.\ the $\lambda_0 \in \Lambda_\Q$ in Remark \ref{rem:zeta-properties} is $\frac{n}{2}$. On the other hand, $\dim \mathcal{N}_\pi \leq 1$; equality holds if and only if $\pi \simeq \tau \hat{\boxtimes} \check{\tau}$, in which case $\mathcal{N}_\pi$ is spanned by the matrix coefficient map $v \otimes \check{v} \mapsto \lrangle{\check{v}, \tau(\cdot)v}$.
	
	For $\lambda \in \CC$, $\Re(\lambda) \gg 0$, the zeta integral in \eqref{eqn:ordinary-zeta} becomes
	\[ \int_{D^\times} \lrangle{\check{v}, \tau(x)v} |\Nrd(x)|^{\frac{n}{2} + \lambda} \xi_0(x) \dd\mu \]
	where $\xi_0$ is an ordinary Schwartz function on $D$ and $\mu$ is a Haar measure on $D^\times$. As is well-known, it points to the standard local $L$-factor $L(\lambda + \frac{1}{2}, \tau)$. By setting $\lambda=0$ (or equivalently, by replacing $\tau$ by $\tau \otimes |\Nrd|^\lambda$), we get the $L$-factor evaluated at its axis of symmetry. Working with half-densities thus leads to a natural normalization of zeta integrals. This has been observed in \cite{Li15}.
\end{example}

\section{Action by Schwartz space}\label{sec:Schwartz-action}
The constructions below will be crucial for \S\ref{sec:meromorphy}. In what follows, we fix a connected reductive group $G$ and work with $\R$-varieties.

Consider a smooth affine variety $X$ in general. The space of scalar-valued Schwartz functions $\Schw(X)$ is defined in \cite{AG08}. It is a Fréchet space topologized by the semi-norms $\xi \mapsto \sup_{X(\R)} |D \xi|$, where $D$ ranges over the Nash differential operator on $X(\R)$ (see \textit{loc.\  cit.}) The finiteness of these semi-norms captures the idea of ``rapid decay''; when $X \simeq \R^n$, this coincides with the classical version. One can also define Schwartz functions valued in some Nash vector bundle $\mathscr{L}$ over $X(\R)$ or over some connected components thereof. For us, the only relevant non-scalar cases are
\begin{compactitem}
	\item $X = G$ and $\mathscr{L} = \mathscr{L}_G$ is the density line bundle. The bundle is $G$-equivariantly trivializable by choosing a Haar measure;
	\item $(G, \rho, X)$ is as in Axiom \ref{axiom:PVS} and $\mathscr{L} = \mathscr{L}_X^{1/2}$. The bundle is $G$-equivariantly trivializable by Lemma \ref{prop:density-trivialization}.
\end{compactitem}

Write $\Schw(G)$ for the $\mathscr{L}_G$-valued Schwartz space. It forms a (non-unital) Fréchet convolution algebra $\Schw(G)$, the \emph{Schwartz algebra} of $G(\R)$; see \cite{dC91}. The automorphism $g \mapsto g^{-1}$ of $G$ induces a continuous anti-involution
\[ \Schw(G) \to \Schw(G), \quad \Xi \mapsto \check{\Xi}. \]
The Schwartz algebra acts continuously on SAF representations of $G(\R)$ by the vector-valued integral $\pi(\Xi)v = \int_{G(\R)} \Xi(g) \pi(g)v$ for all $\Xi \in \Schw(G)$, see \cite[\S 2.5]{BK14}. Given any irreducible SAF representation $\pi$ and a maximal compact subgroup $K \subset G(\R)$, it is known that
\[ V_\pi = \pi(\Schw(G)) V_\pi^{K\text{-fini}}; \]
see for example \cite[p.46]{BK14}.

Suppose $X$ is a smooth affine $G$-variety. Then $\Schw(G)$ acts upon the scalar-valued Schwartz space $\Schw(X)$. This is certainly well-known. Due to the lack of reference, we supply a proof which is inspired by \cite{dC91} for the case $G=X$.

\begin{lemma}\label{prop:Schwartz-action-X}
	Let $X$ be a smooth affine $G$-variety. Then $\Schw(G)$ acts on the Schwartz space $\Schw(X)$ of $X(\R)$ defined in \cite{AG08} via
	\[ (\Xi \xi)(x) = \int_{G(\R)} \Xi(g) \xi(xg), \quad \Xi \in \Schw(G), \; \xi \in \Schw(X) \]
	and the action map is jointly continuous.
\end{lemma}
\begin{proof}
	We have to bound $D \cdot \Xi\xi$ for every Nash differential operator $D$ on $X(\R)$. By \cite[Corollary 4.1.3]{AG08}, it suffices to treat the case of algebraic differential operators $D$.

	The first step is to notice $|\Xi\xi(x)| \leq \|\Xi\|_{L^1(G(\R))} \cdot \sup_{X(\R)}|\xi|$ for all $x$, and note that $\|\cdot\|_{L^1(G(\R))}$ is a continuous semi-norm of $\Schw(G)$.

	Write $X \xleftarrow{\text{pr}_1} X \times G \xrightarrow{a} X$ for the projection and action morphisms. The sheaf of algebraic differential operators $\mathscr{D}_X$ is $G$-equivariant, namely we are given a $G$-linearization $\varphi: a^*(\mathscr{D}_X) \rightiso \text{pr}_1^*(\mathscr{D}_X)$ compatibly with $a^* \mathscr{O}_X \rightiso \text{pr}_1^* \mathscr{O}_X$. Now we may write $\varphi(a^* D) = \sum_{i=1}^m a_i \otimes b_i$, where $a_i$ are algebraic differential operators on $X$, and $b_i$ are regular functions on $G$. Use $D(x) = D(xgg^{-1})$ to deduce
	\begin{align*}
		\left| (D \cdot \Xi\xi)(x) \right| & = \left| \sum_{i=1}^m \int_{g \in G(\R)} b_i(g^{-1}) \Xi(g) a_i(xg) \xi(xg) \right| \\
		& \leq  \sum_{i=1}^m \sup_{X(\R)}\left| \check{b}_i \Xi \cdot a_i \xi \right| < +\infty \qquad \because\text{the first step},
	\end{align*}
	by noting that $\check{b}_i \Xi \in \Schw(G)$ and $a_i \xi \in \Schw(X)$. Since multiplication by $a_i$ (resp. $\check{b}_i$) is continuous on the Schwartz space, that estimate also implies the separate continuity of $(\Xi, \xi) \mapsto \Xi\xi$, thus the joint continuity since both spaces are Fréchet.
\end{proof}

Now we revert to the setting of \S\ref{sec:zeta-integral}. Thus $\Schw(X)$ is $\mathscr{L}_X^{1/2}$-valued, but this matter can be trivialized. The estimates from Theorem \ref{prop:asymptotic} show that $\Schw(G)$ can act on $u = \eta(v)$ for $v \in V_\pi$ (cf. \cite[11.1]{BK14}). More generally $\Schw(G)$ acts on $\eta_\lambda(v)$, where
\begin{align*}
	\eta_\lambda: \pi \otimes |\omega|^\lambda & \longrightarrow C^\infty(X^+(\R)) \\
	v & \longrightarrow |f|^\lambda \eta(v)
\end{align*}
is seen to be a continuous intertwining operator. So we have $\eta_\lambda(\pi(\Xi) v) = \Xi \eta_\lambda(v)$ whenever $\Xi \in \Schw(G)$ and $v \in V_\pi$.

\begin{lemma}\label{prop:Schwartz-transpose}
	In the range of convergence for zeta integrals, we have
	\[ Z_\lambda(\eta, \pi(\Xi)v, \xi) = Z_\lambda(\eta, v, \check{\Xi}\xi) \]
	for all $\Xi \in \Schw(G)$, $v \in V_\pi$ and $\xi \in \Schw(X)$.
\end{lemma}
\begin{proof}
	Express the zeta integrals using \eqref{eqn:flat-zeta}.	By Fubini's theorem, it suffices to show the integrability of
	\[ \Xi(g) \eta(v)_0(xg) |f(xg)|^{\lambda - \lambda_0} |\det\rho(g)| \xi_0(x) \]
	over $(g,x) \in G(\R) \times X(\R)$. Here we use a translation-invariant measure on $X(\R)$; the factor $|\det\rho(g)|$ comes from the Jacobian of $g$ acting on $X$.

	In the proof of Theorem \ref{prop:zeta-convergence} we saw $\eta(v)_0 |f|^{\lambda-\lambda_0}$ can be bounded by a positive Nash function on $X(\R)$, and that function can in turn be bounded by a positive polynomial $q$ on $X$ (see Lemma \ref{prop:Nash-bound}). By algebraicity, $q(xg) = \sum_{i=1}^m a_i(x) b_i(g)$ for some $a_i \in \R[X]$, $b_i \in \R[G]$. It remains to show that $\Xi(g)|\det\rho(g)| b_i(g)$ and $\xi_0(x) a_i(x)$ are integrable over $G(\R)$ and $X(\R)$ for $i=1, \ldots, m$, respectively. This holds true since $\Xi$ and $\xi_0$ are Schwartz functions.
\end{proof}

\section{Proof of meromorphic continuation}\label{sec:meromorphy}
Retain the notation from \S\ref{sec:zeta-integral}. Fix an irreducible SAF representation $\pi$ of $G(\R)$ and let $u := \eta(v) \in C^\infty(X^+)$ for $v \in V_\pi$ and $\eta \in \mathcal{N}_\pi$. For studying meromorphic continuation, it is safe to assume $f_i \geq 0$ upon replacing $f_i$ by $f_i^2$ and drop the $|\cdot|$ from \S\ref{sec:zeta-integral}. Put
\[ f_\natural := f_1 \cdots f_r \]
which yields a non-negative function on $X(\R)$, and $\{ f_\natural = 0\} = \partial X$.

\begin{lemma}\label{prop:u-estimate}
	Fix a basis $w_1, \ldots, w_l$ of $X$ and the corresponding vector fields $D_1, \ldots, D_l$ on $X$. For each $(i_1, \ldots, i_l) \in \Z_{\geq 0}^l$, there exists $k \in \Z$ such that
	\[ \forall x \in X^+(\R), \quad \left| D_1^{i_1} \cdots D_l^{i_l} u\right| \leq p(x) \cdot f_\natural(x)^{-k} \]
	for some Nash function $p \geq 0$ on $X(\R)$. Consequently, the strengthened form \eqref{eqn:cond-u-strong} of Hypothesis \ref{hyp:cond-u} holds on $X^+(\R)$.
\end{lemma}
\begin{proof}
	As the first step, we show by induction on $i_1 + \cdots + i_l$ that
	\begin{gather}\label{eqn:diff-lemma-aux}
		D_1^{i_1} \cdots D_l^{i_l} u \;\in\; \text{span}_\R \bigg\langle h u' : \; h \in \R[X][f_\natural^{-1}], \; u' \in \eta(V_\pi) \bigg\rangle.
	\end{gather}
	This is trivial when $i_1 = \cdots = i_l = 0$. For the induction step, consider $D_j (h u')$, where $1 \leq j \leq n$. Choose a basis $v_1, \ldots, v_n$ of $\mathfrak{g}$ and denote the corresponding vector fields on $X$ as $D_{v_1}, \ldots, D_{v_n}$. Use Lemma \ref{prop:diff-op-relation} to write $D_j = \sum_{i=1}^n a_i D_{v_i}$ for some $a_1, \ldots, a_n \in \R[X][f_\natural^{-1}]$. Now
	\begin{align*}
		D_j (h u') & = (D_j h) u' + h \cdot D_j u' \\
		& =  (D_j h) u' + \sum_{i=1}^n h a_i D_{v_i} u'.
	\end{align*}
	Note that $D_{v_i} u' = \eta( \pi(v_i) v') \in \eta(V_\pi)$ if $u' = \eta(v')$; also, $D_j$ leaves $\R[X][f_\natural^{-1}]$ stable. The proof of \eqref{eqn:diff-lemma-aux} is complete.
	
	It follows that $D_1^{i_1} \cdots D_l^{i_l} u = f_\natural^{-k} (t_1 u_1 + \cdots + t_m u_m)$ for some $t_i \in \R[X]$ and $u_i \in \eta(V_\pi)$. Each $u_i$ is bounded by a Nash function on $X^+(\R)$ by Theorem \ref{prop:asymptotic}; furthermore, Lemma \ref{prop:Nash-bound} says any Nash function on $X^+(\R)$ can be bounded by some $f_\natural^\mu p_1$ where $\mu \in \Z$ and $p_1 \geq 0$ is a polynomial function on $X(\R)$. The first assertion is thus established.
	
	To check \eqref{eqn:cond-u-strong}, observe that $D_1, \ldots, D_l$ generates $\mathscr{T}_X$.
\end{proof}

Choose any maximal compact subgroup $K \subset G(\R)$.

\begin{theorem}\label{prop:meromorphy}
	Given $\eta \in \mathcal{N}_\pi$ and $v \in V_\pi^{K\text{-fini}}$, the zeta integral $Z_\lambda(\eta, v, \cdot)$ in Definition \ref{def:zeta-integral} can be extended to a meromorphic family in $\lambda \in \Lambda_\CC$ of tempered distributions. Furthermore, take $a \in \Lambda_\Z$ with $a >_X 0$.
	\begin{enumerate}[(i)]
		\item There exist affine hyperplanes $H_1, \ldots, H_t \subset \Lambda_\CC$ whose vectorial parts $\vec{H}_i$ are defined over $\Q$, such that the pole set of $Z_\lambda$ is a union of translates $H_i - m a$, for various $i$ and $m \in \Z_{\geq 1}$.
		\item There exists a holomorphic function $\lambda \mapsto L(\eta, \lambda)$ on $\Lambda_\CC$ such that
			\[ LZ_\lambda(\eta, v, \xi) := L(\eta, \lambda) Z_\lambda(\eta, v, \xi) \;\text{is holomorphic on}\; \Lambda_\CC, \quad \forall v \in V_\pi^{K\text{-fini}}, \; \forall \xi \in \Schw(X). \]
			Furthermore, one may take $L(\eta, \lambda) = \prod_{i=1}^m \Gamma(\alpha_i(\lambda))^{-1}$ where $\alpha_1, \ldots, \alpha_m$ are affine functions on $\Lambda_\CC$ whose gradients are among $\{\vec{H}_1, \ldots, \vec{H}_m \}$. % It is NOT NECESSARILY bounded over vertical strips.
%		\item The map $\lambda \mapsto LZ_\lambda(\eta,v,\xi)$ above is NOT NECESSARILY bounded in vertical strips, for each $(v,\xi)$.
	\end{enumerate}
\end{theorem}
\begin{proof}
	By Corollary \ref{prop:u-holonomicity-symm}, $u := \eta(v)$ generates a holonomic $D_{X^+_\CC}$-module. The machine in \S\ref{sec:holonomicity} is thus applicable in view of Lemma \ref{prop:u-estimate}. In particular, Theorem \ref{prop:multivariate-b} gives the meromorphic continuation of $\lambda \mapsto Z_\lambda(\eta,v, \cdot)$ with values in $\Schw(X)^\vee$. For the chosen $v$, the description of poles in (i) is a standard consequence of the existence of $b$-functions, see \cite[Théorème A.3]{BD92}; in particular, the orders of poles are uniformly bounded by the degree of $b$. These properties are well-known when $r=1$.

	Consider (ii) for $v \in V_\pi^{K\text{-fini}}$. The description of (i) implies that one can choose affine functions $\alpha_1, \ldots, \alpha_m$ according to the configuration of singular hyperplanes such that $L(\lambda) := \prod_{i=1}^m \Gamma(\alpha_i(\lambda))^{-1}$ kills all the poles of $Z_\lambda(\eta, v, \cdot)$. The point is to choose an $L$ that works for all $v \in V_\pi^{K\text{-fini}}$. This is based on the observations below.
	\begin{itemize}
		\item If $L(\lambda) \cdot u|f|^\lambda$ extends holomorphically to $\Lambda_\CC$, the same holds for every element in $L(\lambda) \cdot D_{X_\CC} u|f|^\lambda$. % by the differential calculus of tempered distributions.
		\item Every $w \in \mathfrak{g}$ gives rise to a vector field on $X(\R)$, identifiable with an element of $D_{X_\CC}$. Consequently, every $L(\lambda)$ that works for $v$ also works for all elements in $\pi(\mathcal{U}(\mathfrak{g}_\CC)) v$.
		\item Since $V_\pi^{K\text{-fini}}$ is an irreducible Harish-Chandra module, it is a finitely generated $\mathcal{U}(\mathfrak{g}_\CC)$-module by \cite[Theorem 4.3]{BK14}. Thus we obtain a desired $L$.
	\end{itemize}
\end{proof}

In order to treat the non $K$-finite vectors, we shall appeal to the \emph{principle of analytic continuation} from Gelfand--Shilov \cite[Chapter I, A.2.3]{GS1}, also stated in \cite[pp.65--66]{Ig00}. We present a mild generalization here.

\begin{proposition}\label{prop:GS-principle}
	Let $E$ be a barreled topological vector space and denote by $\Hom(E,\CC)$ its abstract dual. Given a map $T: \CC^r \to \Hom(E, \CC)$, $\lambda \mapsto T_\lambda$. Suppose $T$ satisfies
	\begin{itemize}
		\item for each $v \in E$, the function $\lambda \mapsto T_\lambda(v)$ is holomorphic on $\CC^r$;
		\item there exists an open subset $U \neq \emptyset$ of $\CC^r$ such that $T$ restricts to a holomorphic map $U \to E^\vee$.
	\end{itemize}
	Then $T$ is actually a holomorphic map $\CC^r \to E^\vee$.
\end{proposition}
In other words, continuity of $T_\lambda$ propagates from $\lambda \in U$ to all of $\CC^r$. Note that \cite{Ig00} allows more general domains than $\CC^r$.
\begin{proof}
	Fix $\lambda_0 \in U$. For every $v \in E$, we have $T_\lambda(v) = \sum_{\bm{k} \geq 0} c_{\bm{k}}(v) (\lambda - \lambda_0)^{\bm{k}}$ for uniquely determined $c_{\bm{k}}(v) \in \CC$. Here we adopt the notation of multi-indices $\bm{k} = (k_1, \ldots, k_r)$. Then $c_{\bm{k}}(v)$ is linear in $v$.

	By \cite[Corollary 2 to Theorem 34.2]{Tr67}, $E^\vee$ is quasi-complete with respect to the topology of pointwise convergence. Quasi-complete means: every bounded closed subset is complete. Take $\epsilon > 0$ so small that $\{ \lambda: \forall i \; |\lambda_i - \lambda_{0,i}| \leq \epsilon \}$ is contained in $U$. Put $C := \{ \lambda: \forall i \; |\lambda_i - \lambda_{0,i}|=\epsilon \}$. With the topology of pointwise convergence on $E^\vee$, holomorphy implies that $T: C \to E^\vee$ is continuous. The theory of Gelfand--Pettis integrals for quasi-complete spaces (see \cite{Gar14}) is applicable and we may define the elements of $E^\vee$
	\[ c_{\bm{k}} := \frac{1}{(2\pi i)^r} \oint_C \dfrac{T_\lambda}{(\lambda - \lambda_0)^{\bm{k} + \bm{1}}} \dd\lambda, \quad \bm{1} := (1, \ldots, 1). \]
	The characterization of Gelfand--Pettis integrals together with Cauchy's formula inside $U$ entail that $c_{\bm{k}}$ maps any $v \in E$ to the previously defined $c_{\bm{k}}(v)$. Our notation is thus consistent.
	
	Now for every $\lambda \in \CC$, we have the pointwise limit $\sum_{\bm{k} \leq \bm{n}} c_{\bm{k}} (\lambda - \lambda_0)^{\bm{k}} \to T_\lambda$ as $\bm{n} \to +\infty$; the left-hand side always lies in $E^\vee$. Applying \cite[Corollary to Theorem 33.1]{Tr67}, a consequence of the Banach--Steinhaus theorem, we see $T_\lambda \in E^\vee$. This establishes the holomorphy of $T$.
\end{proof}

\begin{theorem}\label{prop:meromorphy-1}
	Retain the notations above. The family $LZ_\lambda$ has a unique extension to
	\[ LZ_\lambda: \Lambda_\CC \longrightarrow \mathrm{Bil}(V_\pi, \mathcal{X}) = (V_\pi \hat{\otimes} \Schw(X))^\vee. \]
	% and it is NOT NECESSARILY bounded over vertical strips for fixed $(v,\xi)$.
	Moreover, the bilinear form $(v,\xi) \mapsto LZ_\lambda(\eta,v,\xi)$ on $(\pi \otimes |\omega|^\lambda) \times \Schw(X)$ is $G(\R)$-invariant.
\end{theorem}
\begin{proof}
	Use the results from \S\ref{sec:Schwartz-action}. To define $LZ_\lambda(\eta,v,\xi)$ for all $\lambda \in \Lambda_\CC$ and $v \in V_\pi$, we put
	\begin{equation}\label{eqn:LZ-ext-Schwartz}\begin{gathered}
		LZ_\lambda\left( \eta, \sum_{i=1}^m \pi(\Xi_i) v_i , \xi \right) := \sum_{i=1}^m LZ_\lambda\left( \eta, v_i, \check{\Xi}_i \xi \right), \\
		\Xi_i \in \Schw(G), \; v_i \in V_\pi^{K\text{-fini}}, \; i=1,\ldots,m.
	\end{gathered}\end{equation}
	Recall that $\pi(\Schw(G)) V_\pi^{K\text{-fini}} = V_\pi$.
	\begin{asparaenum}[(a)]
		\item The right-hand side of \eqref{eqn:LZ-ext-Schwartz} is holomorphic in $\lambda$ by Theorem \ref{prop:meromorphy}. Let us show it depends solely on $v = \sum_i \pi(\Xi_i) v_i$. If $\sum_i \pi(\Xi_i) v_i = \sum_i \pi(\Xi'_i) v'_i$ in $V_\pi$, then for $\Re(\lambda) \gg_X 0$ Lemma \ref{prop:Schwartz-transpose} entails
		\[\begin{tikzcd}[row sep=small, column sep=small]
			Z_\lambda\left( \eta, \sum_i \pi(\Xi_i) v_i, \xi \right) \arrow[-,double equal sign distance]{r} & \sum_i Z_\lambda\left( \eta, \pi(\Xi_i) v_i, \xi \right) \arrow[-,double equal sign distance]{r} & \sum_i Z_\lambda\left( \eta, v_i, \check{\Xi}_i \xi \right) \\
			Z_\lambda\left( \eta, \sum_i \pi(\Xi'_i) v'_i, \xi \right) \arrow[-,double equal sign distance]{u} \arrow[-,double equal sign distance]{r} & \sum_i Z_\lambda\left( \eta, \pi(\Xi'_i) v'_i, \xi \right) \arrow[-,double equal sign distance]{r} & \sum_i Z_\lambda\left( \eta, v'_i, \check{\Xi}'_i \xi \right),
		\end{tikzcd}\]
		every term being interpreted by Theorem \ref{prop:zeta-convergence}. The equality between the rightmost two terms holds and extends analytically to all $\lambda \in \Lambda_\CC$, and we can multiply all terms by $L(\eta,\lambda)$. This reasoning also shows that \eqref{eqn:LZ-ext-Schwartz} is compatible with Theorem \ref{prop:zeta-convergence} in the range of convergence.
		\item To obtain $LZ_\lambda: \Lambda_\CC \to \mathrm{Bil}(V_\pi, \mathcal{X}) = (V_\pi \hat{\otimes} \Schw(X))^\vee$, it remains to show that $LZ_\lambda(\eta,v,\xi)$ is jointly continuous in $(v, \xi) \in V_\pi \times \Schw(X)$ for each $(v,\xi)$. Since $V_\pi$ and $\Schw(X)$ are both Fréchet, it suffices to check separate continuity. The continuity in $\xi$ is easier: in view of \eqref{eqn:LZ-ext-Schwartz}, it results from Theorem \ref{prop:meromorphy} and the continuity of $\Schw(G)$-action on $\Schw(X)$ in Lemma \ref{prop:Schwartz-action-X}.
		\item In order to obtain continuity in $v = \sum_i \pi(\Xi_i) v_i$ for $\xi$ fixed, recall that
		\begin{compactitem}
			\item continuity is known inside the range of convergence, and
			\item for fixed $(v,\xi)$, we just showed that $LZ_\lambda(\eta,v,\xi)$ is holomorphic on $\Lambda_\CC$.
		\end{compactitem}
		Now we can apply the Proposition \ref{prop:GS-principle} to propagate the continuity in $v$ to all $\Lambda_\CC$. This is legitimate since $V_\pi$ is Fréchet, thus barreled.
		\item We have arrived at the holomorphy of $LZ_\lambda: \Lambda_\CC \to \mathrm{Bil}(V_\pi, \Schw(X))$. Such an extension is unique since $V_\pi^{K\text{-fini}}$ is dense in $V_\pi$.
		\item Finally, the $G(\R)$-invariance has been observed in Remark \ref{rem:zeta-properties} in the range of convergence. The general case follows by analytic continuation.
	\end{asparaenum}
\end{proof}

This confirms the predictions in \cite[\S\S 4.3--4.4]{Li15} in the setting of Axiom \ref{axiom:PVS}. For well-chosen $L(\eta,\lambda)$, one might expect some connection with the inverse of Langlands' Archimedean $L$-factors.

\printbibliography[heading=bibintoc]	% For Biblatex

\vspace{1em}
\begin{flushleft}
	Wen-Wei Li \\
	E-mail address: \href{mailto:wwli@math.ac.cn}{\texttt{wwli@math.ac.cn}} \\
	Academy of Mathematics and Systems Science, Chinese Academy of Sciences \\
	55, Zhongguancun donglu, 100190 Beijing, People's Republic of China. \\ \vspace{0.7em}
	University of Chinese Academy of Sciences \\
	19A, Yuquan lu, 100049 Beijing, People's Republic of China.
\end{flushleft}

\end{document}